\documentclass{amsproc}
\usepackage[dvipsnames, svgnames]{xcolor}
\usepackage{amssymb, enumitem, tikz-cd, multirow, nccmath}

\title{Gluing genus 1 and genus 2 curves along \texorpdfstring{$\ell$}{ℓ}-torsion}

\author{Pitchayut Saengrungkongka}
\address{
  Department of Mathematics,
  Massachusetts Institute of Technology,
  Cambridge,
  MA 02139,
  USA
}
\email{psaeng@mit.edu}
\author{Noah Walsh}
\address{
  Department of Mathematics,
  Massachusetts Institute of Technology,
  Cambridge,
  MA 02139,
  USA
}
\email{nwalsh21@mit.edu}
\subjclass[2020]{11G30, 14H25, 14H40}

\usepackage[normalem]{ulem}

\usepackage[style=alphabetic, backend=biber, sorting=nyt, doi=false, isbn=false, url=false, hyperref,backref,backrefstyle=none, maxnames=20, maxalphanames=20]{biblatex}
\newbibmacro{string+doiurlisbn}[1]{%
  \iffieldundef{doi}{%
    \iffieldundef{url}{%
      \iffieldundef{isbn}{%
        \iffieldundef{issn}{%
          #1%
        }{%
          \href{https://books.google.com/books?vid=ISSN\thefield{issn}}{#1}%
        }%
      }{%
        \href{https://books.google.com/books?vid=ISBN\thefield{isbn}}{#1}%
      }%
    }{%
      \href{\thefield{url}}{#1}%
    }%
  }{%
    \href{https://dx.doi.org/\thefield{doi}}{#1}%
  }%
}
\DeclareFieldFormat{title}{\usebibmacro{string+doiurlisbn}{\mkbibemph{#1}}}
\DeclareFieldFormat[article,incollection,inproceedings,unpublished,misc,book]{title}%
    {\usebibmacro{string+doiurlisbn}{\mkbibquote{#1}}}
\DefineBibliographyStrings{english}{%
  backrefpage = {$\uparrow$},%
  backrefpages = {$\uparrow$}%
}

\usepackage{xurl}
\usepackage[pdfusetitle, pdfencoding=auto, psdextra, breaklinks=true]{hyperref}
\usepackage{xcolor}
\usepackage{color}
\definecolor{mylinkcolor}{rgb}{0.5,0.0,0.0}
\definecolor{myurlcolor}{rgb}{0.0,0.0,0.7}
\hypersetup{
 colorlinks,
 urlcolor=myurlcolor,
 citecolor=myurlcolor,
 linkcolor=mylinkcolor,
 breaklinks=true
}

\usepackage[nameinlink]{cleveref}

\numberwithin{equation}{section}
\newtheorem{theorem}[equation]{Theorem}
\newtheorem{lemma}[equation]{Lemma}
\newtheorem{proposition}[equation]{Proposition}
\newtheorem{corollary}[equation]{Corollary}
\theoremstyle{definition}
\newtheorem{algorithm}[equation]{Algorithm}
\newtheorem{definition}[equation]{Definition}

\newtheorem{example}[equation]{Example}

\theoremstyle{remark}
\newtheorem{remark}[equation]{Remark}

\newcommand{\vocab}[1]{\textbf{\textit{\boldmath #1}}}
\DeclareMathOperator{\Jac}{Jac}
\DeclareMathOperator{\GL}{GL}
\DeclareMathOperator{\SL}{SL}
\DeclareMathOperator{\GSp}{GSp}
\DeclareMathOperator{\Frob}{Frob}
\DeclareMathOperator{\Gal}{Gal}
\DeclareMathOperator{\lcm}{lcm}
\DeclareMathOperator{\Res}{Res}
\DeclareMathOperator{\End}{End}
\DeclareMathOperator{\Sym}{Sym}
\newcommand{\eps}{\varepsilon}
\newcommand{\ZZ}{\mathbb Z}
\newcommand{\QQ}{\mathbb Q}
\newcommand{\CC}{\mathbb C}
\newcommand{\FF}{\mathbb F}
\newcommand{\Qbar}{\overline{\QQ}}
\newcommand{\GalQ}{\Gal(\Qbar/\QQ)}
\newcommand{\ttwomatrix}[4]{\left(\begin{smallmatrix} #1 & #2 \\ #3 & #4\end{smallmatrix}\right)}

\newcommand{\ttwovector}[2]{\left(\begin{smallmatrix} #1 \\ #2\end{smallmatrix}\right)}

\newcommand{\rhoHperpH}{\overline\rho_{H^\perp\!/H}}

\usepackage[color=white, linecolor=red, bordercolor=red]{todonotes}

\begin{document}
\begin{abstract}
    Let $Y$ be a genus $2$ curve over $\QQ$.
    We provide a method to systematically search for possible candidates of a prime $\ell\geq 3$ and a genus $1$ curve $X$ for which there exists a genus $3$ curve $Z$ over $\QQ$ whose Jacobian is, up to quadratic twist, $(\ell, \ell, \ell)$-isogenous to the product of Jacobians of $X$ and $Y$, building on the work by Hanselman, Schiavone, and Sijsling for $\ell=2$. We find several such pairs $(X,Y)$ for prime $\ell$ up to $13$. We also improve their numerical gluing algorithm, allowing us to successfully glue genus $1$ and genus $2$ curves along their $13$-torsion.
\end{abstract}
\maketitle
\section{Introduction}
The study of curves is a central topic in arithmetic geometry. Exhaustive lists of curves have long been a useful tool in number theory, starting with the Antwerp tables that contain elliptic curves up to conductor $200$ \cite{antwerp}, continuing with Cremona's tables of elliptic curves \cite{cremona_table}, and persisting today with the collections of elliptic curves and genus 2 curves in the LMFDB \cite{lmfdb}.

A large number of higher genus curves can be constructed by gluing curves of smaller genera. For any two curves $X$ and $Y$ of genera $g_X$ and $g_Y$, the process of \vocab{gluing} produces a curve $Z$ of genus $g_X+g_Y$ such that the Jacobian of $Z$, denoted $\Jac(Z)$, is isogenous to the product $\Jac(X)\times \Jac(Y)$. In other words, $\Jac(Z)$ is a quotient of $\Jac(X)\times\Jac(Y)$ by a finite subgroup $G$. If $G$ is a subgroup of the product of the $n$-torsion subgroups $\Jac(X)[n]$ and $\Jac(Y)[n]$ for a positive integer $n$, we say that this gluing is \vocab{along $n$-torsion}.

Given generic curves $X$ and $Y$, there are many possible choices of subgroups $G$ that give rise to gluings over $\CC$,
and one can construct such gluings analytically by considering $\Jac(X)$ and $\Jac(Y)$ as complex tori. However, for arithmetic applications, we are interested in curves defined over non-algebraically closed fields such as $\QQ$. If we take curves $X$ and $Y$ randomly, then it is very likely there will be no gluing $Z$ that is defined over $\QQ$. We are interested in systematically producing pairs of curves that admit a gluing over $\QQ$.
These curves often have interesting properties that deviate from generic behavior, such as a special torsion structure, a nontrivial endomorphism ring, or an interesting Sato-Tate group.

\subsection{Previous Work}
The simpler case of gluing two curves of genus $1$ (i.e., $1+1=2$) along $n$-torsion has been studied in many aspects.
\begin{enumerate}[label=(\roman*)]
\item \textbf{Criteria for gluings to exist.} Frey and Kani \cite{frey_kani}  derive the necessary conditions for two curves to admit a gluing along $n$-torsion, namely, if $E_1$ and $E_2$ are gluable, then there exists an antisymplectic, $\GalQ$-equivariant isomorphism $E_1[n]\to E_2[n]$.
\item \textbf{Find all gluable pairs.} Given a list of elliptic curves (e.g., all elliptic curves in the L-functions and modular forms database, LMFDB \cite{lmfdb}),
we are interested in systematically searching for all pairs that are gluable.
When $n=\ell$ is prime, Cremona and Freitas \cite{global_symplectic_type} gave a complete algorithm for detecting all such antisymplectic isomorphisms and applied their methods to all pairs of elliptic curves in the LMFDB. 
\item \textbf{Parametrize all gluable pairs.} Building on (ii),
we ask if one can parameterize all pairs of elliptic curves that are gluable. For a fixed elliptic curve $E_1$, the space of all elliptic curves $E_2$ that are gluable is one-dimensional and is a twist of the modular curve $X(n)$. For $n\in\{2,3,4,5\}$, the space is isomorphic to $\mathbb P^1$. The details have been worked out by Rubin and Silverberg in \cite{ec_gluable_family_2, ec_gluable_family_3_5, ec_gluable_family_4}. 
Fisher \cite{ec_gluable_family_7_11} derives explicit formulas for $n\in\{7,11\}$, but in that case there are finitely many $E_2$'s for a fixed $E_1$.
There are also efforts to parametrize \emph{pairs} of gluable elliptic curve along $n$-torsion for higher $n$: Fisher \cite[Cor.~1.3]{ec_gluable_pairs} shows that there are infinitely many pairs of gluable curves for all $n\leq 10$.
\item \textbf{Computing gluing.} One way method of computing gluings over $\QQ$ is via analytic techniques over $\CC$. In simpler cases, we also have explicit formulas. Howe, Lepr\'evost, and Poonen \cite[Prop.~4]{ec_2_torsion_gluing} gave an explicit formula for $n=2$, and Br\"oker et. al. \cite[\S A.1]{ec_3_torsion_gluing} gave a formula for $n=3$. As $\ell$ grows larger, the explicit formulas quickly become unwieldy.
\end{enumerate}

Genus $1$ plus genus $2$ gluing (i.e., $1+2=3$) has not been studied as widely. Ritzenthaler and Romagny \cite{reverse_1_plus_2} gave a formula for recovering the equation of the genus $2$ factor of a genus $1$ plus genus $2$ gluing along $2$-torsion, provided that the genus $1$ factor is known. Then Hanselman, Schiavone, and Sijsling \cite{1_plus_2_2_torsion_gluing} gave a comprehensive algorithm and explicit formula for gluing along 2-torsion.  
They answer some of the questions above for gluing genus $1$ and $2$ curves as follows.
\begin{itemize}
    \item To answer (i), they \cite[\S 1]{1_plus_2_2_torsion_gluing} provide a concrete criterion for a gluing along $\ell$-torsion to exist, which we recall in \Cref{sec:glue_condition}.
    \item To answer (iv), they \cite[\S 2.1]{1_plus_2_2_torsion_gluing} outline the analytic algorithm to compute the gluing along $\ell$-torsion, which was implemented in \cite{gluing_code}. In the case of $\ell=2$, they also provide an explicit formula to glue genus $1$ and $2$ curves along $2$-torsion.
    \item To answer (iii), if $\ell=2$, for a genus two curve $Y$, Hanselman \cite[\S 2.2]{HanselmanThesis} constructs an infinite family of genus $3$ curves resulting from gluing $Y$ with an elliptic curve along their $2$-torsion.
    
    More generally, for a fixed genus $2$ curve $Y$ such that there exists an elliptic curve $X$ that admits gluing with $Y$ along $\ell$-torsion, one can reduce the problem of finding all gluable elliptic curves $X'$ to finding elliptic curves $X'$ whose $\ell$-torsion is symplectically isomorphic (as a $\GalQ$-module) to the $\ell$-torsion of $X$. See \Cref{subsec:determine_X} for more details.
\end{itemize}

Our work attempts to answer (ii) and improves existing methods in (iv).

\subsection{Our Results}
This paper considers the problem of gluing two curves of genus $1$ and genus $2$ along their $\ell$-torsion to obtain a curve of genus $3$ for a prime $\ell$. We focus on the case of $\ell\geq 3$ that has not been as widely studied. There are two key results.

The first key result is that, given a curve $Y$ of genus $2$, we demonstrate how to systematically and efficiently search for genus $1$ curves $X$ in the LMFDB and corresponding primes $\ell$ such that $X$ and $Y$ are gluable along $\ell$-torsion.
More specifically, given a genus $2$ curve $Y$, we first eliminate all but finitely many $\ell$'s. Then for each such $\ell$, we search for elliptic curves $X$ such that there exists a $\GalQ$-stable subgroup $G\subset\Jac(X)[\ell]\times \Jac(Y)[\ell]$ such that the quotient
\begin{equation}
    Q = \frac{\Jac(X)\times \Jac(Y)}{G}
\end{equation}
is a principally-polarized abelian variety. The search process culminates in the workflow described in \Cref{sec:gluing_algorithm}.
We also provide an algorithm to rigorously verify the existence of such $G$ in the generic case, as detailed in \Cref{alg:proving_isomorphism,alg:symplectic}.
If $Q$ is isomorphic to a Jacobian of some genus $3$ curve $Z$  (which happens generically, when $Q$ is not a product of two or more Jacobians), then curves $X$ and $Y$ produce a gluing $Z$. We use this result to glue a large number of curves in the LMFDB and discover curves with interesting geometric endomorphism rings in \Cref{subsec:endomorphism}.

The second key result is an improvement of the numerical gluing algorithm in \cite[\S 2.1]{1_plus_2_2_torsion_gluing} to obtain a more efficient \Cref{alg:fast_gluing}.
With this algorithm, we computing a gluing along $13$-torsion in $24$ minutes in \Cref{ex:13_gluing}.

\subsection{Organization of the Article}

In \Cref{sec:glue_condition}, we describe the condition under which two curves are gluable, following \cite{1_plus_2_2_torsion_gluing}. Then in \Cref{sec:frob_element}, we explain how to use Frobenius elements to filter pairs of gluable curves efficiently. \Cref{sec:proving_isomorphism} explains how to utilize Serre's modularity conjecture to rigorously verify (in the generic case) a part of the gluability condition. \Cref{sec:symplectic_type} adapts the symplectic test in \cite{local_symplectic_type} to verify the other part of the gluability condition. We put all the pieces together and describe our current workflow in \Cref{sec:gluing_algorithm}, which includes the speedup of the gluing algorithm.
Finally, \Cref{sec:examples} lists some examples produced from running our workflow on curves in the LMFDB.

\subsection{Notations}
For any prime $p$ and rational number $r\neq 0$, let $\nu_p(r)$ be the $p$-adic valuation of $r$, i.e., the exponent of $p$ in the prime factorization of $r$.

For any abelian variety $A$ over a field $k$ and a positive integer $n$, $A[n]$ denotes the set of $n$-torsion points over the algebraic closure $\overline k$.
For any curve $X$, let $N_X$ denote the conductor of $X$, and for any prime $p$ not dividing $N_X$, let $a_{p,X}$ denote the trace of Frobenius at $p$ of $X$.
Additionally, we define $a_{p,X}$ to be $1$, $-1$, or $0$ if $X$ is an elliptic curve with split multiplicative, non-split multiplicative, or additive reduction modulo $p$, respectively.

\subsection*{Acknowledgements}
This research was conducted through the MIT Department of Mathematics’s Summer Program for Undergraduate Research (SPUR). The authors would like to thank our mentors Edgar Costa and Sam Schiavone for their guidance throughout the program. We also thank Eran Assaf and Shiva Chidambaram for helpful discussions. We also thank Prof. David Jerison and Jonathan Bloom for organizing SPUR and for their thoughtful comments about our research. Finally, we thank anonymous referees who provide helpful comments to this paper.

\section{Gluability Conditions}
\label{sec:glue_condition}
Let $n\geq 2$ be an integer. Let $X$ and $Y$ be smooth curves of any genus over a base field $k$ with characteristic not dividing $n$. Informally, a gluing of $X$ and $Y$ is a curve $Z$ with an isomorphism $\Jac(Z) \simeq (\Jac(X) \times \Jac(Y))/G$, where $G$ is a subgroup of $(\Jac(X) \times \Jac(Y))[n]$ for some $n$. Such an isomorphism does not exist for all subgroups $G$; in what follows, we consider some necessary conditions for $(\Jac(X) \times \Jac(Y))/G$ to be isomorphic to a Jacobian and consider how they translate to conditions on $G$.

First, the Jacobian of any curve $C$ comes with a \vocab{principal polarization}. This is an isomorphism $\lambda: \Jac(C) \to \Jac(C)^\vee$, satisfying technical conditions laid out in \cite[\S 1.11]{milne_av}. We will consider more generally necessary conditions for $(A_1 \times A_2)/G$ to be a Jacobian, where $A_1$ and $A_2$ are arbitrary principally polarized abelian varieties. Let $\lambda_1: A_1 \to A_1^\vee$ and $\lambda_2: A_2 \to A_2^\vee$ be these polarizations.

For any abelian variety $A$, a polarization $\lambda: A \to A^\vee$  induces for each $n$ the \vocab{Weil Pairing}
\begin{equation}
e_n^{\lambda} \colon A[n] \times A[n] \to \mu_n,
\end{equation}
an alternating bilinear form taking values in the $n$-th roots of unity. (See \cite[\S 1.13]{milne_av} for the exact definition.)

The pairings $e_n^{\lambda_1}$ and $e_n^{\lambda_2}$ on $A_1$ and $A_2$ induce a pairing
\begin{equation}
\begin{aligned}
    e_n \colon (A_1 \times A_2)[n] \times (A_1 \times A_2)[n] &\to \mu_n\\
    ((P_1, P_2), (Q_1, Q_2)) &\mapsto e_n^{\lambda_1}(P_1, Q_1)\; e_n^{\lambda_2}(P_2, Q_2)
\end{aligned}
\end{equation}
coming from the product polarization 
$
    \lambda_1 \times \lambda_2 : A_1 \times A_2 \to A_1^\vee \times A_2^\vee \simeq (A_1 \times A_2)^\vee.
    $

We only consider cases in which the polarization on the quotient comes from the $n$-th power of this product polarization. (See \cite[Rmk. 1.1.4, Thm.~1.1.10]{HanselmanThesis} for why generically one does not gain anything by considering polarizations other than the $n$-th power of the product.) In order for this to yield a polarization on the quotient, $G$ must be \vocab{isotropic}, i.e., the pairing must vanish on $G\times G$. Furthermore, by \cite[Lem.~2.1]{computing_isogeny}, $G$ must be maximal with respect to this property in order for the polarization to be principal.

The polarization of a Jacobian is indecomposable, so it is also necessary for the polarization on the quotient to be indecomposable.
Therefore, it is also necessary that $G$ be \vocab{indecomposable}, i.e., not of the form $G_1 \times G_2$ for $G_1 \subset A_1$, $G_2 \subset A_2$, as otherwise the polarization will be a product of polarizations on $A_1/G_1$ and $A_2/G_2$.

To summarize, $G$ must be an \vocab{indecomposable maximal isotropic subgroup}. Let us now specialize to the case where $A_1 = \Jac(X)$ and $A_2 = \Jac(Y)$ for curves $X$ and $Y$ of genera $1$ and $2$, respectively. Suppose that $G$ is an indecomposable maximal isotropic subgroup. Then $(\Jac(X) \times \Jac(Y))/G$ is an abelian variety of dimension $3$ defined over some extension $k'$ of $k$. If the polarization of $(\Jac(X) \times \Jac(Y))/G$ is indecomposable, then by \cite{threefolds_are_jacobians}, $(\Jac(X) \times \Jac(Y))/G$ is isomorphic to the Jacobian of some curve $Z$, where $Z$ and the isomorphism are defined over some further extension $k''$ of $k'$. Serre proved in the appendix to \cite{lauter} that one may in fact take $k''$ to be a quadratic extension of $k'$.

We can now formally define a gluing.

\begin{definition}
    \label{def:gluing}
    An \vocab{$(n,n)$-gluing} of the curves $X$ and $Y$ over $k$ is a triple $(Z,\psi, G)$, where $Z$ is a smooth curve over $k$ and a subgroup $G\subseteq \Jac(X)[n]\times\Jac(Y)[n]$ such that along with an isomorphism of principally polarized abelian varieties
    \begin{equation}
    \psi : \frac{\Jac(X)\times \Jac(Y)}{G} \stackrel{\sim}{\longrightarrow} \Jac(Z).
    \end{equation}
\end{definition}

In the situation when $n=\ell$, a prime number, \cite[\S 1]{1_plus_2_2_torsion_gluing} gives the following description of such subgroups.
\begin{proposition} \textnormal{\cite[Prop.~1.18]{1_plus_2_2_torsion_gluing}}
    \label{prop:complex-gluability}
    A subgroup $G\subseteq \Jac(X)[\ell]\times\Jac(Y)[\ell]$ is an indecomposable maximal isotropic subgroup if and only if 
    \begin{equation}
    G = \{(x, y)\,\mid\,\phi(x) = y + H\}
    \end{equation}
    for some one-dimensional subgroup $H \subset G$ and some antisymplectic isomorphism $\phi: \Jac(X)[\ell] \to H^\perp/H$.
\end{proposition}

For most choices of $G$, the resulting abelian variety $(\Jac(X)\times \Jac(Y))/G$ will not be defined over $k$.
In order for $(\Jac(X)\times \Jac(Y))/G$ to be defined over $k$,
it is necessary that $G$ be Galois-stable.
\cite[Prop.~1.39]{1_plus_2_2_torsion_gluing} worked out what this means in terms of $(G,\phi)$ in \Cref{prop:complex-gluability}.

\begin{theorem}
\label{thm:gluability} \textnormal{\cite[Prop.~1.39]{1_plus_2_2_torsion_gluing}}
Following the notations of \Cref{prop:complex-gluability}, $G$ is Galois stable if and only if both of the following holds
\begin{enumerate}
    \item[(i)]\label{thm-cond:H-stability} $H$ is Galois-stable.
    \item[(ii)]\label{thm-cond:phi-equivariance} $\phi$ is Galois-equivariant.
\end{enumerate}
\end{theorem}

Conversely, from the paragraph preceding \Cref{def:gluing}, we have the following converse.
\begin{theorem}
    \label{thm:gluability_converse}
    Suppose that
    \begin{enumerate}[label=(\roman*)]
    \item $G$ is Galois-stable (i.e., satisfies both conditions of \Cref{thm:gluability}); and
    \item the quotient $(\Jac(X)\times\Jac(Y))/G$ is not a product of two or more Jacobians.
    \end{enumerate}
Then there exists a curve $Z$ such that $\Jac(Z)\simeq (\Jac(X)\times\Jac(Y))/G$, and the isomorphism is defined over a quadratic extension of $k$.
\end{theorem}
Condition (ii) of \Cref{thm:gluability_converse} can be verified numerically in terms of a period matrix of $(\Jac(X)\times \Jac(Y))/G$: at most one of its even theta values can vanish.
For most of the paper, we focus on searching for pairs of curves $(X,Y)$ over $\QQ$ for which there exists a Galois-stable maximal isotropic subgroup $G\subset\Jac(X)[\ell]\times\Jac(Y)[\ell]$ (equivalently, $(H,\phi)$ satisfying the conditions of \Cref{thm:gluability}). For this to hold, we need both of the following to be true:
\begin{enumerate}[label=(\roman*)]
\item there exists a $\GalQ$-stable subgroup of $\Jac(Y)[\ell]$.
\item $H^\perp/H$ and $\Jac(X)[\ell]$ are isomorphic as $\GalQ$-modules, and the isomorphism is antisymplectic.
\end{enumerate}
Fix a genus $2$ curve $Y$.
In \Cref{sec:frob_element}, we describe how to rule out elliptic curves that do not satisfy (i) or the isomorphic part of (ii).
Then in \Cref{sec:symplectic_type}, we discuss how to test the remaining condition of (ii): that the isomorphism is antisymplectic.

\begin{remark}
    \label{rmk:decomposable_jacobian}
    Even if there exists a Galois-stable maximal isotropic subgroup of $\Jac(X)[\ell]\times \Jac(Y)[\ell]$, there might still not be a gluing.
    Consider the genus $2$ curve $Y$ given by \href{https://www.lmfdb.org/Genus2Curve/Q/2646/b/71442/1}{\texttt{2646.b.71442.1}} on the LMFDB, the elliptic curve $X$ given by \href{https://www.lmfdb.org/EllipticCurve/Q/21/a/6}{\texttt{21.a6}} on the LMFDB, and $\ell=3$.
    In this case, we can show that there exists such subgroup $G$. 
    However, we have a strong numerical evidence that there does not exist a gluing between $X$ and $Y$. More specifically, the Jacobian of the genus 2 curve is $2$-isogenous to a product of two elliptic curves, and quotienting out by $G$ splits the $\Jac(Y)$ factor into a product of two elliptic curves.
    We have yet to investigate a condition to tell a priori whether $(\Jac(X)\times \Jac(Y))/G$ will be a Jacobian or not.
\end{remark}
\section{Rational \texorpdfstring{$\ell$}{ℓ}-torsion Subgroups}
\label{sec:frob_element}
Throughout this section, we fix a genus $2$ curve $Y$ and study for what primes $\ell$ there might exist an $(\ell, \ell)$-gluing, and if so, whether there exists such a genus $1$ curve $X$ for which $X$ and $Y$ admit an $(\ell, \ell)$-gluing.

It turns out that for a fixed $Y$, the condition (i) of \Cref{thm:gluability} already rules out the possible primes $\ell$ to a finite set. We explain the reason in \Cref{subsec:finding_ell}. Once we know $\ell$, we present an algorithm that computes the trace of a Frobenius element acting on $H^\perp/H$ in \Cref{subsec:frob_action}. We explain how this constrains $X$ in \Cref{subsec:frob_traces}.

We now introduce the concept of a Frobenius polynomial. For any prime $p\neq\ell$ and genus two curve $Y$ for which $Y$ has good reduction $\overline Y$ modulo $p$, let $K$ be a number field such that $\Jac(Y)$ has a full $\ell$-torsion over $K$. Pick any prime ideal $\mathfrak p$ above $p$, and let $\Frob_p\in\Gal(K/\QQ)$ 
be the Frobenius element corresponding to $\mathfrak p$%
\footnote{Although $\Frob_p$ is defined only up to conjugacy class, the choice of which $\Frob_p$ we pick will not matter.}.
Then $\FF_{\mathfrak p} := \mathcal O_K/\mathfrak p$ is a finite field, and $\Frob_p$ acts on points of $\Jac(Y)$ 
in a way that
\begin{equation}
    \Frob_p((x:y:z)) \equiv  (x^p:y^p:z^p)\pmod{\mathfrak p}
\end{equation}
Thus, we get an action of $\Frob_p$ on $\Jac(\overline Y)[\ell]\simeq\FF_\ell^4$ by a linear map in $\GL_4(\FF_\ell)$ with characteristic polynomial congruent modulo $\ell$ to the \vocab{Frobenius polynomial}, which is of the form
\begin{equation}
    \label{eq:frob_poly_Y}
    F_{Y,p}(T) := T^4 - a_{p,Y}T^3 + a'_{p,Y}T^2 - pa_{p,Y}T + p^2 \in \ZZ[x],
\end{equation}
independent of $\ell$.
The coefficients $a_{p,Y}$ and $a'_{p,Y}$ are determined by the number of points of $Y$ over $\FF_p$ and $\FF_{p^2}$. See \cite[Chapter II]{milne_av} for more details.
\subsection{Finding Possible Primes}
\label{subsec:finding_ell}
The condition that $\Jac(Y)[\ell]$ must have a $1$-dimensional Galois-stable subspace $H$ already restricts the possible values of $\ell$ to a finite set, even without knowing the curve $X$. Let $\mathcal L$ be a set of primes for which $H$ exists. An algorithm to determine a finite superset of $\mathcal L$ is studied in \cite[\S 3.1]{computing_isogeny}, using Dieulefait's criterion in  \cite[\S 3.1]{dieulefait_criterion}.
Their algorithm only filters out $\ell$'s not dividing the conductor $N_Y$ and does not do anything for those $\ell$'s that divide $N_Y$.
Thus, we modify their algorithm so that it can rule out some primes $\ell$ dividing $N_Y$ as well.

The action of an element $\sigma\in\GalQ$ on $\Jac(Y)[\ell]\in\FF_\ell^4$ can be expressed as a matrix in $\GL_4(\FF_\ell)$. This induces a representation $\overline\rho_Y : \GalQ \to \GL(\Jac(Y)[\ell])$ of dimension $4$. 
The $1$-dimensional $\GalQ$-stable subspace $H$ induces two $1$-dimensional representations $\GalQ\to\FF_\ell^\times$:
\begin{itemize}
    \item on $H$ itself; we let this representation correspond to the character $$\eps_1 : \GalQ\to\GL(H) \simeq \FF_\ell^\times.$$
    \item on $\Jac(Y)[\ell] / H^\perp$ (by Galois-equivariance of the Weil pairing); we let this representation correspond to the character 
    $$\eps_2 : \GalQ \to \GL(\Jac(Y)[\ell]/H^\perp) \simeq \FF_\ell^\times.$$
\end{itemize}
By Galois-equivariance of the Weil pairing,
we have $\eps_1 \eps_2 = \chi_\ell$, where $\chi_\ell$ is the mod-$\ell$ cyclotomic character (the unique character such that $\sigma(\zeta_\ell) = \zeta_\ell^{\chi_\ell(\sigma)}$). 
Because $\FF_\ell^\times$ is abelian, by the Kronecker-Weber theorem, $\eps_i$ ($i=1,2$) must factor through $\Gal(\QQ(\zeta_e)/\QQ)\simeq (\ZZ/e\ZZ)^\times$ for some positive integer $e$ (where $\zeta_e$ is a primitive $e$-th root of unity). The smallest such $e$ is the \vocab{conductor} of $\eps_i$.

\begin{proposition}
    \label{prop:finite_conductor}
    Let $N_Y$ be the conductor of $Y$, and let
    $d$ be the largest integer such that $d^2$ divides $N_Y/\ell^{\nu_\ell(N_Y)}$. Define
    \begin{equation}
    \label{eq:conductor_bound}
    D = \begin{cases}
        \ell d & \text{if }\ell \text{ divides } N_Y, \\
        d & \text{if }\ell \text{ does not divide } N_Y.
    \end{cases}
\end{equation}
    Then the conductor of at least one of $\eps_1$ and $\eps_2$ divides $D$.
\end{proposition}
\begin{proof}
    First, we show in general that the conductor of $\eps_1$ and $\eps_2$ both divide $\ell d$.
    Let $I_\ell = \Gal(\Qbar_\ell / \QQ_\ell^{\text{unr}})$
    be the inertia group.
    By \cite[\S 2.3]{serre}, there exist $\alpha$ and $\beta$ such that if $\eps_1' = \eps_1\chi_\ell^{-\alpha}$ and $\eps_2' = \eps_2\chi_\ell^{-\beta}$, then
    both $\eps_1'$ and $\eps_2'$ are characters from $\GalQ\to\FF_\ell^\times$ unramified at $\ell$.
    Since $\eps_1\eps_2 = \chi_\ell$, we must have 
    $\alpha+\beta\equiv 1\pmod{\ell-1}$
    and $\eps_1' \eps_2'$ is the trivial character.
    In particular, $\eps_1'$ and $\eps_2'$ have the same conductor; let it be $k$.
    
    Therefore, $k^2$ divides the conductor of the mod-$\ell$ representation $\overline\rho_Y$, so $k^2$ divides $N_Y/\ell^{\nu_\ell(N_Y)}$, so $k$ divides $d$.
    Thus, the conductors of $\eps_1'$ and $\eps_2'$ both divides $d$, which implies that the conductors of both $\eps_1$ and $\eps_2$ divides $\ell d$.

    Now, assume $\ell$ does not divide $N_Y$. By Raynaud's result \cite[Cor.~3.4.4]{raynaud}, we have $\{\alpha,\beta\} = \{0,1\}$. Therefore, at least one of $\eps_1$ and $\eps_2$ is unramified away from $\ell$, implying that the conductor of $\eps_1$ or $\eps_2$ is the same as that of $\eps_1'$ or $\eps_2'$.
\end{proof}

\Cref{prop:finite_conductor} shows that at least one of $\eps_1$ and $\eps_2$, say $\eps$, factors through $\Gal(\QQ(\zeta_D)/\QQ) \simeq (\ZZ/D\ZZ)^\times$, inducing a Dirichlet character $\chi : (\ZZ/D\ZZ)^\times \to \FF_\ell^\times$. This narrows down the possible candidates for $\chi$ to a finite set.

Furthermore, for any prime $p$, the image of $\Frob_p$ in $\Gal(\QQ(\zeta_D)/\QQ)\simeq (\ZZ/D\ZZ)^\times$ is $p$. Thus, if $f$ is the order of $p$ in $(\ZZ/D\ZZ)^\times$, then
\begin{equation}
\eps\left(\Frob_p\right)^f = \eps\big(\Frob_p ^f\big) = \eps\left(p^f\right) = \eps(1) = 1.
\end{equation}
Moreover, $\eps(\Frob_p)$ is an eigenvalue of $\Frob_p$ (with eigenvector in $H$). This means that the polynomial $T-\eps(\Frob_p)$ divides $F_{Y,p}(T)$. In other words, $F_{Y,p}(T)$ and $T^f-1$ have a common root in $\FF_\ell$, and so the resultant $\Res(F_{Y,p}(T), T^f-1)$ is divisible by $\ell$. This gives rise to the following algorithm.
\begin{algorithm}
    \label{alg:H_test}
    \emph{Input:} a curve $Y$ of genus $2$, its conductor $N_Y$, and a finite nonempty set of primes $\mathcal P$, each of which gives good reduction of $Y$.

    \emph{Output:} a finite superset of $\mathcal L$, the set of primes $\ell$ for which there exists a one-dimensional $\GalQ$-stable subgroup $H\subset\Jac(X)[\ell]$.

    \begin{enumerate}
    \item Let $d$ be the largest positive integer such that $d^2\mid N_Y$.
    \item For each prime $p\in\mathcal P$, do the following:
    \begin{itemize}
        \item compute the Frobenius polynomial $F_{Y,p}(T)$; and
        \item compute the order $f_p$ of $p$ in $(\ZZ/d\ZZ)^\times$.
    \end{itemize}
    \item Let $\mathcal L_{\text{good}}$ be the set of primes dividing $\gcd_{p\in\mathcal P} \Res(F_{Y,p}(T), T^{f_p}-1).$
    \item For each prime $\ell$ dividing $N_Y$, we run the following test for each prime $p\in \mathcal P$:
    \begin{itemize}
        \item compute the Frobenius polynomial $F_{Y,p}(T)$ and the order $g_p$ of $p$ in $(\ZZ/\ell d\ZZ)^\times$; and
        \item check whether $F_{Y,p}(T)$ and $T^{g_p}-1$ have a common root in $\FF_\ell$.
    \end{itemize}
    Let $\mathcal L_{\text{bad}}$ be the set of primes $\ell$ dividing $N_Y$ that pass the above test for all primes $p$.
    \item Return $\mathcal L = \mathcal L_{\text{good}}\cup\mathcal L_{\text{bad}}$.
    \end{enumerate}
\end{algorithm}

The discussion preceding this algorithm shows that \Cref{alg:H_test} indeed returns a valid superset of $\mathcal L$. Such a superset is always finite: by the Riemann hypothesis for curves, all roots of $F_{Y,p}$ have absolute value $\sqrt p$, so $\Res(F_{Y,p}(T), T^{f_p}-1)\neq 0$ for all primes $p$, and hence $\mathcal L_{\text{good}}$ is always finite.

\subsection{The Frobenius Action on \texorpdfstring{$H^\perp/H$}{Hperp/H}}
\label{subsec:frob_action}
Assuming that $H$ exists, we now proceed to study the condition (ii) of \Cref{thm:gluability}. We defer the study of the symplectic condition to \Cref{sec:symplectic_type} and focus on the condition that $H^\perp/H$ and $\Jac(X)[\ell]$ are isomorphic as $\GalQ$-modules.

For any prime $p$ at which $Y$ has good reduction, 
the Frobenius element $\Frob_p$
acts on $H^\perp/H \simeq \FF_\ell^2$ by a matrix in $\GL_2(\FF_\ell)$.
We define
\begin{equation}
    \label{eq:trace_Hperp_def}
    b_p = \text{trace of } \Frob_p \text{ on } H^\perp/H.
\end{equation}
The traces $b_p$ are important data describing the Galois action on $H^\perp/H$, which will allow us to search for a suitable elliptic curve $X$. See \Cref{prop:frob_traces} for more details.
In this subsection, we will explain how to narrow down the possible values of $b_p$.

Recall from \Cref{subsec:finding_ell} that the action of $\GalQ$ on $H$ is determined by a character $\eps : \GalQ \to \FF_\ell^\times$, 
and $\eps$ factors through $\Gal(\QQ(\zeta_D)/\QQ)\simeq (\ZZ/D\ZZ)^\times$, 
where $D$ is defined in \labelcref{eq:conductor_bound}. 
Thus, $\eps$ corresponds to a Dirichlet character $\chi : (\ZZ/D\ZZ)^\times \to \FF_\ell^\times$,
and we have $\eps(\Frob_p) = \chi(p)$.

For any prime $p$, the action of $\Frob_p$ onto
the one-dimensional subspace $H$ and $\Jac(Y)[\ell]/H^\perp$ 
must be multiplication by $\chi(p)$ and $p/\chi(p)$ in some order,
which are two eigenvalues of $\Frob_p$ acting on $\Jac(Y)$.
The trace on $H^\perp/H$ must be the sum of 
the remaining two eigenvalues. Thus,
\begin{equation}    
    \label{eq:frob_trace_recover}
    b_p = a_{p,Y} - \chi(p) - \tfrac{p}{\chi(p)} \pmod\ell.
\end{equation}

Hence, if one knows $\chi$, one can determine $b_p$ for all $p$.
For a fixed $Y$, there are finitely many possibilities 
for $\chi : (\ZZ/D\ZZ)^\times \to \FF_\ell^\times$.
We rule out $\chi$ by checking that, for any prime $q$
not dividing $N_Y$, the number $\chi(q)$ is a root of $F_{Y,q}(T)$.
This leads to the following algorithm.

\begin{algorithm}
    \label{alg:frob_trace_recover}
    \emph{Input.} A genus $2$ curve $Y$, its conductor $N_Y$, a prime $\ell$, and a finite set of primes $\mathcal Q$.
    
    \emph{Output.} A finite list of candidate functions that take a prime number $p\neq\ell$ and output $b_p$ (defined in \eqref{eq:trace_Hperp_def}) modulo $\ell$.
    
    For each one-dimensional Galois-stable subgroup $H\subset \Jac(Y)[\ell]$, the function that outputs the trace of $\Frob_p$ on $H^\perp/H$ must be in the output list
    (but there might be extraneous functions). 

    \begin{enumerate}
    \item Compute $D$ as defined in \labelcref{eq:conductor_bound}.
    \item Enumerate the set $\mathcal X$ of all Dirichlet characters $(\ZZ/D\ZZ)^\times \to \FF_\ell^\times$.
    \item For each prime $q\in\mathcal Q$, compute $F_{Y,q}(T)$. Remove from $\mathcal X$ any character $\chi$ such that $\chi(q)$ is not a root of $F_{Y,q}(T)$.
    \item For each function $\chi\in \mathcal X$, we return the function
    $$p\ \mapsto\ a_{p,Y} - \chi(p) - \tfrac{p}{\chi(p)}\pmod\ell.$$
    \end{enumerate}
\end{algorithm}
From \labelcref{eq:frob_trace_recover}, we can see that the output of this algorithm includes at least one function for each possible $H$. If there are multiple possible $H$'s, the algorithm returns multiple functions, each corresponding to a candidate $H$.

\subsection{Frobenius Traces of Gluable Elliptic Curve}
\label{subsec:frob_traces}
For this entire section, fix a genus $2$ curve $Y$
and a prime $\ell\geq 3$ such that there exists
a one-dimensional $\GalQ$-stable subgroup
$H\subset \Jac(Y)[\ell]$.
For any prime $p$ at which $Y$ has good reduction,
let $b_p$ denote the trace of $\Frob_p$ in $H^\perp/H$
(which we have narrowed down the possibilities 
in \Cref{alg:frob_trace_recover}).

Suppose that $X$ is an elliptic curve such that $\Jac(X)[\ell]$
and $H^\perp/H$ are isomorphic as $\GalQ$-modules.
By comparing $b_p$ with the trace of Frobenius at $p$ of $X$ (which we denoted $a_{p,X}$), we can rule out some of such $X$, as detailed in the following proposition.
\begin{proposition}
    \label{prop:frob_traces}
    Let $X$ be an elliptic curve such that 
    $\Jac(X)[\ell]$ and $H^\perp/H$ are isomorphic as $\GalQ$-modules.
    Let $p\neq\ell$ be a prime such that $Y$ has good reduction modulo $p$.
    \begin{enumerate}[label=(\roman*)]
    \item If $X$ has good reduction modulo $p$, then $b_p\equiv a_{p,X}\pmod\ell$.
    \item If $X$ has split multiplicative reduction modulo $p$, then $b_p \equiv 1+p\pmod{\ell}$.
    \item If $X$ has non-split multiplicative reduction modulo $p$, then $b_p \equiv -(1+p)\pmod{\ell}$.
    \item If $\ell\geq 5$, then $X$ cannot have additive reduction modulo $p$.
    \end{enumerate}
\end{proposition}
\begin{proof}
\hspace{0pt}
\begin{enumerate}[label=(\roman*)]
\item The Frobenius element $\Frob_p$ acts on $H^\perp/H$ and on $X_{\FF_p}[\ell]$ by the same matrix in $\GL_2(\FF_\ell)$ up to a change of basis. Thus, they have the same trace modulo $\ell$.
\item By the theory on Tate's curve \cite[Thm.~V.5.3.]{silverman_advanced}, there exists $q\in\QQ_p$ such that $X_{\overline{\QQ}_p} \simeq \overline{\QQ_p}^\times/q^{\ZZ}$ as groups.
This isomorphism is Galois equivariant.
By considering the $\ell$-torsion of both sides,
we find that $X_{\overline\QQ_p}[\ell] \simeq \langle q^{1/\ell}, \zeta_\ell\rangle$
(where $\zeta_\ell$ is the $\ell$-root of unity).
Therefore, the Frobenius element $\Frob_p$ acts on
$X_{\FF_p}[\ell]$ by matrix $\ttwomatrix 1*0p$,
so it must act on $H^\perp/H$ by the same matrix
in modulo $\ell$.
Hence, we conclude that $b_p\equiv 1+p\pmod \ell$.
\item There exists a quadratic twist $X'$ of $X$ 
such that $X'$ has a split multiplicative reduction modulo $\ell$.
By (ii), the trace of $\Frob_p$ acting on $X'_{\FF_p}[\ell]$ is $1+p$.
Thus, the trace of $\Frob_p$ acting on $X_{\FF_p}[\ell]$ is $-(1+p)$,
which must be equal to $b_p$ in modulo $\ell$.
\item If $p\neq\ell\geq 5$ and $X$ has additive reduction modulo $p$, then we claim that the torsion field $\QQ(X[\ell])$ is ramified at $p$.
To see this, following the notations of \cite[Chapter VII]{silverman}, let $X_0$ and $X_1$ denote the points in $X_{\QQ_p}$ that reduce to non-singular point and the additive identity $\widetilde O$ in $X_{\FF_p}$, respectively. 
By \cite[Thm.~VII.6.1]{silverman}, the size of the group $X_{\QQ_p}/X_0$ is at most $4$. 
Thus, for $\ell \geq 5$, we have $X_{\QQ_p}[\ell] \subseteq X_0$.
However, if $\QQ(\Jac(X)[\ell])$ were unramified at $p$, 
$X_1$ would not have any $\ell$-torsion, so by \cite[Prop.~2.1]{silverman}, one would have $(\ZZ/\ell \ZZ)^2$ as a subgroup of $(X_{\overline\FF_p})_{\text{ns}} \simeq (\overline{\FF}_p, +)$, a contradiction.

Thus, the Galois representation $\overline\rho_X : \GalQ \to \GL_2(\FF_\ell)$ is ramified at $p$.
However, the Galois representation $\rhoHperpH : \GalQ \to \GL_2(\FF_\ell)$ is isomorphic to $\overline\rho_X$ but is unramified at $p$, a contradiction. \qedhere
\end{enumerate}
\end{proof}
\begin{corollary}
    \label{cor:hasse_weil_bound}
    Suppose $p\neq \ell$ and $\ell\geq 5$. If
    \begin{equation}b_p \notin \{-(p+1), p+1\} \cup [-2\sqrt p, 2\sqrt p],
    \end{equation}
    then there is no elliptic curve $X$ such that $H^\perp/H$ and $\Jac(X)[\ell]$ are isomorphic as $\GalQ$-modules.
\end{corollary}
\begin{proof}
    Follows from \Cref{prop:frob_traces} and Hasse-Weil bound.
\end{proof}
\Cref{cor:hasse_weil_bound} can be used to prove that some genus $2$ curves have no gluable elliptic curves. We provide an example for $\ell\in\{11,13,19\}$.
\begin{example}
    When $\ell=11$,
    The curve $Y$ with LMFDB label  \href{https://www.lmfdb.org/Genus2Curve/Q/353/a/353/1}{\texttt{353.a.353.1}} and equation $y^2 + (x^3 + x + 1)y = x^2$ has rational $11$-torsion point (so $H$ exists), and its Frobenius polynomial at $2$ is
    \begin{align*}F_{Y,2}(T) 
    &= T^4+T^3+3T^2+2T^3+4 = (T-1)(T-2)(T^2+4T+2).
    \end{align*}
    Following \Cref{alg:frob_trace_recover}, we have $\chi\equiv 1$, so we can compute $b_2 = -4\pmod{11}$.
    Thus, $Y$ has no gluable elliptic curves.

    Similarly, the curves with LMFDB labels \href{https://www.lmfdb.org/Genus2Curve/Q/349/a/349/1}{\texttt{349.a.349.1}} and \href{https://www.lmfdb.org/Genus2Curve/Q/169/a/169/1}{\texttt{169.a.169.1}} have no gluable elliptic curves with $\ell=13$ and $\ell=19$, respectively.
\end{example}
\begin{remark}
When $\ell\in\{3,5\}$, by \cite[Thm.~3]{modularity_mod_5},
there always exists an elliptic curve $X$ such that
$H^\perp/H$ and $\Jac(X)[\ell]$ are isomorphic as $\GalQ$-modules.
Furthermore, one can make it isomorphic with a correct symplectic type \cite[\S 13]{ec_hessian}.
Thus, there always exists a gluable elliptic curve provided that $\ell\in\{3,5\}$ and $H$ exists.

We do not yet know whether there is an example of a genus $2$ curve for which $H$ exists when $\ell=7$, but there is no gluable elliptic curve.
\end{remark}

\section{Proving Isomorphism of Galois Representations}
\label{sec:proving_isomorphism}
In this section, we utilize Serre's modularity conjecture to study the action of $\GalQ$ on $H^\perp/H$ more closely. The result of this is that, if the Galois representation associated to $H^\perp/H$ is irreducible, then there is an algorithm to rigorously prove that $H^\perp/H$ and $\Jac(X)[\ell]$ are isomorphic as $\GalQ$-modules.

We begin by reviewing the statement of Serre's modularity conjecture in \Cref{subsec:serre_conjecture} and use it to show that the representation corresponding to $H^\perp/H$ is modular in \Cref{subsec:modularity}.
After this, the algorithm is divided into two steps.
\begin{enumerate}
    \item Proving that there indeed exists a one-dimensional  $\GalQ$-stable subgroup $H\subset \Jac(Y)[\ell]$. We do this by utilizing a numerical algorithm from analytic description of $Y$. This is detailed in \Cref{subsec:proving_H}.
    \item Proving the isomorphism between $H^\perp/H$ and $\Jac(X)[\ell]$. We do this by combining modularity of the representation corresponding to $H^\perp/H$ and Sturm's bound.
    This gives an explicit bound on how many Frobenius traces to check, which is explained in \Cref{subsec:irreducible}.
\end{enumerate}
Finally, in \Cref{subsec:reducible_case}, we provide some comments about when the reducible case.

\subsection{Serre's Modularity Conjecture}
\label{subsec:serre_conjecture}
In this section, we briefly review the statement of Serre's modularity conjecture, which was proven by Khare and Wintenberger in \cite{modularity}.

For any positive integer $N$, we have the congruence subgroup
\begin{equation}
    \Gamma_0(N) = \left\{\ttwomatrix abcd \in \SL_2(\ZZ): c\equiv 0\pmod N\right\}\ \subset\ \SL_2(\ZZ).
\end{equation}
For any Dirichlet character $\eps : (\ZZ/N\ZZ)^\times \to \CC^\times$ and any positive integer $k$, we have the (finite-dimensional) $\CC$-vector space $S_k(N,\eps)$ of cusp forms in $\Gamma_0(N)$ with weight $k$ and nebentypus $\eps$.
We also denote $S_2(N) := S_2(N,\chi_{\text{triv}})$, where $\chi_{\text{triv}}(a)=1$ for all $a$.

Each function $f\in S_2(N,\chi)$ has a \vocab{$q$-expansion} $f(z) = \sum_{n\geq 1}a_nq^n$, where $q=e^{2\pi iz}$. 
For any prime $\ell$, one can reduce modular forms modulo $\ell$ by taking any mapping from $\Qbar$ to $\overline{\FF}_\ell$ and reducing the $q$-expansion coefficients along the mapping. This gives a \vocab{modulo $\ell$-modular form} $\overline f = \sum_{n\geq 1}\overline a_n q^n$ where $\overline a_n \in \overline{\FF}_\ell$. 

We now state Serre's modularity conjecture.
\begin{theorem}[Serre's modularity conjecture, \cite{modularity}]
    Suppose $\overline\rho : \GalQ \to \GL_2(\FF_\ell)$ satisfies
    \begin{itemize}
    \item $\overline\rho$ is irreducible; and
    \item $\overline\rho$ is odd, i.e., $\det \overline\rho(c) = -1$, where $c$ is the complex conjugation map.
    \end{itemize}
    Then there exist positive integers $N=N(\overline\rho)$, $k=k(\overline\rho)$, a character $\eps=\eps(\overline\rho) : (\ZZ/N\ZZ)^\times \to \CC^\times$, and a modulo $\ell$ modular form $f\in\FF_\ell[[q]]$ arising from $S_k(N,\eps)$ with coefficients in $\FF_\ell$ such that
    $a_p \equiv \operatorname{tr}\rho(\Frob_p)\pmod p$
    for all $p$ not dividing $N$.
\end{theorem}
\cite{serre} gives an explicit recipe for determining $N(\overline\rho)$, $\eps(\overline\rho)$, and $k(\overline\rho)$. We will not reproduce the full definitions here, but it is important to note that they satisfy $\det\overline\rho = \eps (\chi_\ell)^{k-1}$ and $k\in [2,\ell^2+1]$.

\subsection{Modularity of the Representation}
\label{subsec:modularity}
Let $Y$ be a genus $2$ curve and $\ell$ be a prime such that there exists a Galois-stable subgroup $H\subseteq\Jac(Y)[\ell]$.
Let $V_Y = \Jac(Y)[\ell]$.
We now use Serre's modularity conjecture to study the $\GalQ$-module structure of $H^\perp/H$.

Since $H$ is $\GalQ$-stable, by Galois-equivariance of the Weil pairing, we deduce that $H^\perp$ is $\GalQ$-stable. Thus, the action of $\GalQ$ onto $H^\perp/H$ is well-defined and thus induces the $2$-dimensional $\bmod\,\ell$-representation $\rhoHperpH : \GalQ \to \GL(H^\perp/H)$.
First, we have the following result about determinant.
\begin{proposition}
    \label{prop:det_Hperp_H}
    For any $\sigma\in\GalQ$, we have
    \begin{equation}
    \det(\rhoHperpH(\sigma)) = \chi_\ell(\sigma),
    \end{equation}
    where $\chi_\ell : \GalQ\to\FF_\ell$ is the cyclotomic character, determined by the action of the primitive $\ell$-th root of unity. 
\end{proposition}
\begin{proof}
    Consider the $\Gal(\Qbar/\QQ)$-invariant filtration
    $0 \subset H \subset H^\perp \subset V_Y$,
    which has successive quotients are $H$, $H^\perp/H$, and $V_Y/H^\perp$. Therefore, we deduce that
    \begin{equation}
    \det(\sigma|_H) \det(\sigma|_{H^\perp/H}) \det(\sigma|_{V_Y/H^\perp}) = \det(\sigma|_{V_Y}) = \chi_\ell(\sigma)^2.
    \end{equation}
    However, by Galois equivariance of the Weil pairing, $\det(\sigma|_H) \det(\sigma |_{V_Y/H^\perp}) = \chi_\ell(\sigma)$, so it follows that $\det(\rhoHperpH(\sigma)) = \det(\sigma|_{H^\perp/H}) = \chi_\ell(\sigma)$.
\end{proof}
Plugging the complex conjugation map $c$ into \Cref{prop:det_Hperp_H}, we see that $\det \rhoHperpH(c) = \chi_\ell(c) = -1$, so $\rhoHperpH$ is odd.
Thus, if we assume that $\rhoHperpH$ is irreducible, then by Serre's modularity conjecture and the previous lemma,
$\rhoHperpH$ is modular with trivial nebentypus.
The next proposition then restricts the level.
\begin{proposition}
    If $\rhoHperpH$ is irreducible, then the level $N(\rhoHperpH)$ divides $N_Y$.
\end{proposition}
\begin{proof} 
    We claim that the conductor divides the conductor of the representation $\overline\rho : \Gal(\Qbar/\QQ)
     \to \GL(V_Y)$, which in turn divides the conductor of $Y$.

    To prove this, it suffices to show that $\nu_p(N(\rhoHperpH))
    \leq \nu_p(N(\overline\rho_Y))$ for all primes $p$.  
    Recall from \cite{serre} that the exponents of $p$ in the conductors
    are defined by
    \begin{align}
    \nu_p(N(\rhoHperpH)) &= \sum_{i\geq 0} \frac{|G_i|}{|G_0|}
    \dim (H^\perp/H)^{G_i} \\
    \nu_p(N(\overline\rho_Y)) &= \sum_{i\geq 0} \frac{|G_i|}{|G_0|}
    \dim (V_Y)^{G_i},
    \end{align}
    where $G_i$ are $p$-adic ramification groups 
    (with lower numbering)
    and $V^G$ is the subspace of $V$ fixed by $G$.

    Thus, it suffices to show that for any subgroup $G\subset\Gal(\Qbar/\QQ)$,
    we have $\dim (H^\perp/H)^G \leq \dim (V_Y)^G$.
    
    Consider the $\Gal(\Qbar/\QQ)$-invariant filtration
    $0 \subset H \subset H^\perp \subset V_Y$,
    which has successive quotients $H$, $H^\perp/H$, and $V_Y/H^\perp$, respectively. For any subgroup $G\subset\GalQ$, we have
    \begin{equation}
        \dim H^G + \dim(H^\perp/H)^G + \dim (V_Y/H^\perp)^G = \dim(V_Y)^G,
    \end{equation}
    so $\dim(H^\perp/H)^G \leq \dim (V_Y)^G$, as desired.
\end{proof}
We now determine the weight of the representation.

\begin{theorem}
    \label{thm:weight}
    Assume that $\rhoHperpH$ is irreducible (hence modular) and $\ell$ does not divide $N_Y$. Then $\rhoHperpH$ is modular with weight $2$.
\end{theorem}
\begin{proof}
By \cite[Prop.~4]{serre}, it suffices to show that $H^\perp/H$ is a finite $\ZZ_\ell$-group scheme.
Since $\Jac(Y)$ has good reduction modulo $\ell$, it follows that $\Jac(Y)[\ell]$ (and hence both $H$ and $H^\perp$) is a finite $\ZZ_\ell$-group scheme.
Finally, by \cite[Cor.~3.3.6 (1)]{raynaud} (with $e=1$), it follows that $H^\perp/H$ is a finite $\ZZ_\ell$-group scheme.
\end{proof}

\begin{corollary}
    \label{cor:modularity}
    Assume that $\rhoHperpH$ is irreducible. Let
    \begin{equation}
        \label{eq:weight_definition}
        k = \begin{cases}
            2 & \ell\text{ does not divide } N_Y \\
            \ell^2+1 & \ell \text{ divides } N_Y.
        \end{cases}
    \end{equation}
    Then there exists a modular form $f\in S_k(N_Y)$ modulo $\ell$ such that $\overline\rho_{f,\ell} \simeq \rhoHperpH$. 

    In particular, if $f = \sum_{n=1}^\infty c_n q^n$ is the $q$-expansion of $f$, then $c_p \equiv b_p\pmod\ell$ for all primes $p$ not dividing $N_Y$.
\end{corollary}
\begin{proof}
    If $\ell$ does not divide $N_Y$, this is immediate by \Cref{thm:weight}.
    
    Otherwise, if $\ell$ divides $N_Y$, then the (untwisted) weight of $\rhoHperpH$ is in the interval $[2,\ell^2+1]$ and is $\equiv 2\pmod{\ell-1}$.
    By von Staudt–Clausen theorem,
    $\nu_\ell(B_{\ell-1}) = -1$ (where $B_{\ell-1}$
    denotes Bernoulli numbers). Thus, the $q$-expansion of Eisenstein series $E_{\ell-1} = 1 + \frac{2(\ell-1)}{B_{\ell-1}}\sum_{n=1}^\infty \sigma_{\ell-1}(n)q^n$ 
    is congruent to $1$ modulo $\ell$.
    Thus, for any modular form $f\in S_k(N_Y)$ modulo $\ell$, the modular form $E_{\ell-1} f \in S_{k+\ell-1}(N_Y)$, whose $q$-expansion in modulo $\ell$ is the exactly the same as $f$ in modulo $\ell$.
    Thus, regardless of what the weight of $\rhoHperpH$ is, one can increment the weight by $\ell-1$ at a time without changing the coefficients modulo $\ell$.
    Hence, there is a corresponding modular form of weight $\ell^2+1$. This concludes the proof.
\end{proof}

\subsection{Finding rational \texorpdfstring{$H$}{H}}
\label{subsec:proving_H}
In this subsection, we describe a numerical algorithm that can compute the approximate coordinates of $H$, if it exists.

Recall that the Jacobian of a curve of genus $2$ over $\CC$ is isomorphic to a complex torus $\mathbb{C}^2/\Lambda$ for some four-dimensional lattice $\Lambda$.
This isomorphism can be constructed explicitly. 
To do so, recall that points on the $\Jac(Y)$ can be parametrized by $\{P_1,P_2\}\in \Sym^2 Y$.
For two chosen base points $Q_1,Q_2\in\Jac(Y)$, we define the \vocab{Abel-Jacobi map}
\begin{equation}
\begin{aligned}
\operatorname{AJ}_{Q_1,Q_2} : \Sym^2 Y &\to  H^0(Y, \Omega) / H_1(Y, \ZZ) = \CC^2 / \Lambda\\
\{P_1, P_2\} &\mapsto \left(\int_{Q_1}^{P_1} \omega_1 + \int_{Q_2}^{P_2} \omega_1,\quad \int_{Q_1}^{P_1} \omega_2 + \int_{Q_2}^{P_2} \omega_2\right),
\end{aligned}
\end{equation}
where one picks a basis for $\omega_1, \omega_2$ for differentials $H^0(Y, \Omega)$; a working basis is $\omega_1 = \tfrac{dx}y$ and $\omega_2 = \tfrac{x\,dx}y$ if $y^2=f(x)$ is a hyperelliptic model of $Y$.
This map does not depend on the choice of a path, as we are considering their values modulo $\Lambda = H_1(Y, \ZZ)$. For more details about this map, we refer the reader to \cite[\S 1.18]{milne_av}. For our purposes, we take $Q_1$ and $Q_2$ to be $\infty$ and $\infty$ or $\infty_1$ and $\infty_2$ according to whether we are using an odd model or an even model. This ensures that the set $\{P_1, P_2\}$ is Galois-invariant.

On this torus, the $\ell$-torsion is the $\ell^4$ points in $\frac{1}{\ell}\Lambda/\Lambda$. For a candidate subgroup $H$ we may now calculate coordinates and check whether it is rational as follows.

\begin{algorithm}
\label{alg:calculate_H}
\emph{Input.} A one-dimensional subgroup of the $\ell$-torsion of $\Jac(Y)$, as represented by its coordinates in $\mathbb{C}^2/\Lambda$. 

\emph{Output.} If the algorithm recognizes $H$ as Galois-stable, \texttt{true}. If the algorithm recognizes $H$ as Galois-nonstable, \texttt{false}. Otherwise, \texttt{unknown}.
\begin{enumerate}
\item Let the nonzero points in $H$ be $P_1, P_2, \ldots, P_{\ell - 1}$.
\item For each $i$ from $1$ to $\ell - 1$, pull back $P_i$ to a point on $\Sym^2 Y$. Define
\begin{equation}
    f_i(T) := \begin{cases}
        (T-x_{i,1})(T-x_{i,2}) & \text{if } P_i\mapsto \{(x_{i,1}, y_{i,1}), (x_{i,2}, y_{i,2})\} \in \Sym^2 Y \\
        T-x_i & \text{if } P_i\mapsto \{(x_i, y_i)\} \in \Sym^2 Y.
    \end{cases}
\end{equation}
\item Compute $f(T) := \prod_{i = 1}^{\ell - 1}f_i(T)$.
\item If any coordinate has a recognizably nonzero complex part, output \texttt{false}.
\item Run a rational recognition algorithm over the coefficients of $f$. If the coefficients are successfully recognized as rational numbers, then we verify that this is correct by computing coordinate of $P_1$ (which will lie in an field extension of $\QQ$ of degree at most $4(\ell-1)$)
and check that $\ell P_1=0$.
If this test passes, output \texttt{true}.
Otherwise, output \texttt{unknown}.
\end{enumerate}

\end{algorithm}

This algorithm works because if $H$ is Galois-stable, the list of $x$-coordinates collected in step 2 must be Galois-stable as well, meaning $f(T) \in \QQ[T]$.

This algorithm does not always return \texttt{true} or \texttt{false}, even if the precision is sufficiently high. 
If the precision is sufficiently high and 
the algorithm still returns \texttt{unknown},
then $H$ is probably Galois-nonstable.

\subsection{Proving Isomorphism}
\label{subsec:irreducible}
One notable property of modular forms is that
if the first few coefficients (up to the Sturm bound \cite{original_sturm_bound}) of two modular forms are congruent modulo $\ell$, then all coefficients of those modular forms are congruent.
Combining this fact with the modularity of $\rhoHperpH$
gives an algorithm to deterministically prove the isomorphism
between two mod-$\ell$ Galois representations
$\rhoHperpH$ and $\overline\rho_X$.
The algorithm mirrors \cite[Prop.~4]{semisimplification}, which is used to prove mod-$\ell$ congruences between elliptic curves. We now describe a version of this algorithm in our setting of genus $2$ curves.
\begin{algorithm}
    \label{alg:proving_isomorphism}
    \emph{Input.} 
    \begin{itemize}
   \item An elliptic curve $X$ (with conductor $N_X$),
   \item a prime $\ell\geq 3$,
    \item a genus $2$ curve $Y$ (with conductor $N_Y$), and
    \item the coordinates of a generator of one-dimensional
    $\GalQ$-stable subgroup $H\subset \Jac(Y)[\ell]$
    (given by their minimal polynomials)
    such that $\rhoHperpH$ is irreducible.
    \end{itemize}
    \emph{Output.} \texttt{true} if $\rhoHperpH\simeq\overline\rho_X$ and \texttt{false} otherwise.
    \begin{enumerate}[label=(\arabic*)]
    \item Compute the level $M$ and the Sturm's bound $B$ by
    \begin{equation}
    \begin{split}
        \label{eq:M}
        M &= \lcm\Bigg(N_X,\ N_Y \prod_{\substack{\text{prime } p \\ p\mid N_Y}} p\Bigg) \quad \text{and}  \quad
        B =\tfrac{k}{12} M \prod_{\substack{\text{prime } p \\ p\mid M}} \left(1-\tfrac 1p\right),
    \end{split}
    \end{equation}
    where $k$ is defined in \eqref{eq:weight_definition}.
    \item For each prime $p\leq B$ not dividing $N_Y$, do the following.
    \begin{itemize}
        \item Compute the action of $\Frob_p$ acting on $H$, which will be a mutiplication by some $t\in\FF_\ell^\times$. This can be done in a finite extension of $\FF_p$.
        \item Compute $b_p$ (defined in \eqref{eq:trace_Hperp_def}) by (cf. \eqref{eq:frob_trace_recover})
        $b_p = a_{p,Y} - t - \frac pt$.
        \item If $p$ does not divide $N_X$, check whether
        $b_p\equiv a_{p,X}\pmod{\ell}$.
        \item If $X$ has a multiplicative reduction modulo $p$,
        check whether $b_p a_{p,X}\equiv p+1\pmod{\ell}$.
    \end{itemize}
    Return \texttt{true} if the condition holds for all $p \leq B$ and \texttt{false} otherwise.
    \end{enumerate}
\end{algorithm}
This algorithm is only feasible
when the conductor of $Y$ is very small: if the conductor $N_Y$ is squarefree, then the algorithm requires checking at least $N_Y^2$ traces. Most genus two curves have conductor at least $1000$, so this algorithm is only feasible for curves with very small conductor.

\begin{remark}
    To certify irreducibility of $\rhoHperpH$, we do the following: first enumerate the set $\mathcal X$ of characters $\eps : \GalQ \to \FF_\ell^\times$ that has conductor dividing $D$ (where $D$ is in \Cref{prop:finite_conductor}).
    Then for each prime $p\leq B$ for which $Y$ has good reduction, eliminate any character $\eps\in \mathcal X$ such that $b_p\not\equiv \eps(p) + \frac{p}{\eps(p)}\pmod{\ell}.$
    If it happens that $\mathcal X=\emptyset$, we can conclude that $H$ is irreducible.
\end{remark}
\begin{proof}[Proof of Correctness of \Cref{alg:proving_isomorphism}]
    If $\rhoHperpH\simeq\overline\rho_X$, then by \Cref{prop:frob_traces}, the conditions in step (4) are all satisfied, so the algorithm returns \texttt{true}.
    Thus, it suffices to show that if the conditions in step (4) are all satisfied, then $\rhoHperpH\simeq\overline\rho_X$.

    By the Modularity Theorem \cite[Thm.~A]{ec_modularity}, let $\overline f\in S_2(\Gamma_0(N_X))$ be the modular form corresponding to $\rho_X$.
    By the weight-incrementing argument in the proof of \Cref{cor:modularity}, there is a modular form $f\in S_k(\Gamma_0(N_X))$ that has the same reduction modulo $\ell$ as $\overline f$.
    By \Cref{cor:modularity}, let
    $g\in S_k(\Gamma_0(N_Y))$ be the modular form corresponding to $\rho_{H^\perp/H}$.
    
    For any modular form $h$, we let $N(h)$ be the level of 
    $h$ (i.e., the smallest positive  integer such that 
    $h\in S_k(\Gamma_0(N(h)))$).
    Recall that the corresponding L-function of $h$ is 
    $L_h(s) = \sum_{n\geq 1} c_n n^{-s}.$
    Define
    $F_{p,h}(T) = \sum_{n=0}^\infty c_{p^n} T^n.$
    If $h$ is an eigenform, then by properties of
    Hecke operators (see \cite[Thm.~5.9.2]{diamond_shurman}),
    then we have
    \begin{equation}
        L_h(s) = \prod_{p\text{ prime}} 
        F_{p,h}(T),\quad F_{p,h}(T) 
        = \begin{cases}
            (1 - c_pT)^{-1} & \text{if }p\mid N \\
            (1 - c_pT + pT^2)^{-1} & \text{if }p\nmid N.
        \end{cases}
    \end{equation}
    
    Let $R_d : S_k(\Gamma_0(N)) \to S_k(\Gamma_0(Nd))$ denote the operator that takes a function $h$ to another function $\tau\mapsto h(d\tau)$. Note that the L-function corresponding to $R_dh$ is $d^{-s} L_h(s)$. 
    Thus, for any polynomial $P$, the L-function corresponding to $P(R_d)h$ is $P(d^{-s}) L_h(s)$.
    By comparing individual L-factors, we have
    \begin{align}
        L_{P(R_p)h} &= \prod_{q\text{ prime}} F_{q,P(R_p)h}(T),\quad \text{and} \\
        \label{eq:R_p}
        F_{q, P(R_p) h}(T)  &=
        \begin{cases}
            F_{q,h}(T) & \text{if }q\neq p \\
            P(T)\cdot F_{p,h}(T) & \text{if }q=p
        \end{cases}.
    \end{align}
    For each prime $p\leq B$, we define operators $U_p$
    that will be applied to $f$ and $V_p$ that will be applied to $g$ so that 
    \begin{enumerate}[label=(\roman*)]
    \item the L-factors of $U_pf$ and $V_pg$ at $p$ are equal,
    \item the L-factors of $U_pf$ and $V_pg$ at all primes 
    $q\neq p$ remain unchanged, and
    \item for any prime $p$, $\nu_p(U_pf)$ and $\nu_p(V_pg)$ are at most $\nu_p(M)$.
    \end{enumerate}
    ($U_p$ and $V_p$ can depend on $f$ and $g$.)
    We define those by splitting into four cases.
    \begin{enumerate}
    \item \textbf{If $p$ divides $N_Y$,}
    then we define
    $$U_p = F_{p,f}(R_p)\quad\text{and}\quad V_p = F_{p,g}(R_p),$$
    so by applying \eqref{eq:R_p} twice, we have 
    \begin{align*}
    F_{p, U_pf}(T) &= F_{p,f}(T)\cdot F_{p,f}(T)^{-1} = 1, \\
    F_{p,V_pg}(T)
    &=F_{p,g}(T)\cdot F_{p,g}(T)^{-1} =1,
    \end{align*}
    verifying (i).
    Moreover, $\deg F_{p,f} = \max(2-\nu_p(N_X), 0)$ by considering reduction types of $X$,
    so 
    $$\nu_p(N(U_pf)) \leq \nu_p(N_X) + \deg F_{p,f} = \max(2, \nu_p(N_X)).$$
    Similarly, $\deg F_{p,g} = 1$, so we have
    $$\nu_p(N(V_pg)) \leq \nu_p(N_Y)+1.$$
    Therefore, both $\nu_p(N(U_pf))$ and $\nu_p(N(V_pg))$
    are at most $\nu_p(M)$, verifying (ii).
    \item \textbf{If $p$ does not divide either $N_X$ or $N_Y$.} 
    Then we have $F_{p,f}(T) = (1-a_{p,X}T + pT^2)^{-1}$ and $F_{p,g}(T) = (1-b_pT + pT^2)^{-1}$, which are automatically congruent modulo $\ell$. 
    Thus, we can define $U_p=V_p=1$, so the levels of $U_pf$ and $V_pg$ are not divisible by $p$.
    Both (i) and (ii) are then satisfied.
    \item \textbf{If $p$ does not divide $N_Y$ and $\nu_p(N_X)=1$,}
    then we have
    $F_{p,f}(T) = (1-a_{p,X}T)^{-1}$ and $F_{p,g}(T) = (1-b_pT + T^2)^{-1}$.
    The fact that $a_{p,X}=\pm 1$, and the $a_{p,X} b_p\equiv p+1\pmod{\ell}$
    implies that $1-a_{p,X}T$ divides $1-b_pT + pT^2$
    in $\FF_\ell[T]$.
    Let the quotient be $1-cT$ for $c\in\FF_\ell$. Then we define
    $$U_p = 1\quad\text{and}\quad V_p = 1-cR_p,$$
    so that by \eqref{eq:R_p}, we have 
    $F_{p,U_pf}(T) = G_{p, U_pf}(T) = (1-a_{p,X}T)^{-1}$,
    verifying (i).
    Moreover, $\nu_p(N(U_pf)) = \nu_p(N_X) = 1$
    and $\nu_p(N(V_p g)) \leq 1$,
    verifying (ii).
    \item \textbf{If $p$ does not divide $N_Y$ and $\nu_p(N_X)\geq 2$,}
    then the L-factors at $p$ of $f$ and $g$ are
    $1$ and $(1-b_pT+T^2)^{-1}$. Thus, we define
    $$U_p = 1\quad\text{and}\quad V_p = 1-b_pR_p+R_p^2,$$
    so from \eqref{eq:R_p}, we have $F_{p,U_pf}(T) = 
    F_{p, U_pg} = 1$, verifying (i).
    Moreover, $\nu_p(N(U_pf)) = \nu_p(N_X)$
    and $\nu_p(N(U_pg)) \leq 2$, both are at most $\nu_p(M)$.
    \end{enumerate}
    Now, let $U = \prod_{p\leq B} U_p$ and $V=\prod_{p\leq B} V_p$.
    Hence, we have that 
    \begin{itemize}
        \item For all primes $p\leq B$, the L-functions corresponding to $Uf$ and $Vg$ can be expressed as an Euler product, and
        the L-factors at $p$ of $Uf$ and $Vg$ are equal.
        \item Both $Uf$ and $Vg$ are in $S_k(\Gamma_0(M))$ for $M$ defined in \eqref{eq:M}.
    \end{itemize} 
    Thus, the coefficients of $1,q, q^2, \dots,q^B$ in the $q$-expansions of the modular forms $Uf$ and $Vg$ are congruent modulo $\ell$.
    Hence, by Sturm's theorem \cite[Thm.~1]{original_sturm_bound}, we get that $Uf$ and $Vg$ are congruent modulo $\ell$.
    This means that $a_p\equiv b_p\pmod \ell$ for \emph{every} prime $p\nmid M$.
    
    Thus, for any prime $p$, the matrices $\rhoHperpH(\Frob_p)$ and $\overline\rho_X(\Frob_p)$ are conjugate.
    By the Chebotarev's density theorem, for any $\sigma\in\Gal(\overline{\QQ}/\QQ)$,
    one can select $p$ such that $\sigma$ and $\Frob_p$ are conjugate,
    and so $\rhoHperpH(\sigma)$ and $\overline\rho_X(\sigma)$
    are conjugate for any $\sigma\in\Gal(\overline{\QQ}/\QQ)$.
    Thus, by the Brauer–Nesbitt theorem, $\rhoHperpH$ and $\overline\rho_X$ are isomorphic up to semisimplification. However, since we assumed $\rhoHperpH$ is irreducible, it follows that $\rhoHperpH\simeq\rho_X$.
\end{proof}

\subsection{The Reducible Case}
\label{subsec:reducible_case}
If that $\rhoHperpH$ is reducible, we cannot use the algorithm above because $\rhoHperpH$ is not necessarily modular. 
Even if $\rhoHperpH$ were modular, we can only use \Cref{alg:proving_isomorphism} to show that $\rhoHperpH$ and $\rho_X$ are isomorphic up to semisimplification.
An algorithm using modular polynomials similar to \cite[\S 3.6.]{global_symplectic_type} does not extend well to genus $2$ curves due to the unwieldy nature of modular polynomials in genus $2$.
One might need to resort to explicitly computing $\rhoHperpH$.

\section{Checking the Antisymplectic Condition}
\label{sec:symplectic_type}
Throughout this section, let $X$ and $Y$ be curves of genus $1$ and genus $2$, respectively, and let $\ell$ be a prime such that
\begin{itemize}
    \item there is a Galois-stable $1$-dimensional subspace $H\subseteq \Jac(Y)[\ell]$; and
    \item there is a $\GalQ$-module isomorphism $\phi : H^\perp/H \to \Jac(X)[\ell]$.
\end{itemize}
Then the only condition left to verify before we can glue $X$ and $Y$ is that $\phi$ is antisymplectic, i.e., for any $P,Q\in H^\perp/H$, we have $e_\ell(\phi(P), \phi(Q)) = e_\ell(P,Q)^{-1}$. If $\ell=2$, then this condition is tautological. Hence, we assume $\ell$ is odd for the remainder of this section.

If the image of the mod-$\ell$ representation $\rhoHperpH : \GalQ \to \GL(H^\perp/H)$ is sufficiently nice, we may still be able to determine whether this condition holds even before trying to glue those curves. We adapt the method in \cite{local_symplectic_type} to do so.
\subsection{Symplectic Type}

If $\phi : H^\perp/H \to \Jac(X)[\ell]$ is a Galois module isomorphism, then there exists a constant $\alpha\in\FF_\ell^\times$ such that for any $P,Q\in H^\perp/H$, we have 
$e_\ell(\phi(P), \phi(Q)) = e_\ell(P,Q)^\alpha$
because $\alpha$ is simply the determinant of $\phi$
under a suitable basis.
Note that if we replace $\phi$ by $[t]\circ\phi$ where $t\in\ZZ$ (which is still a Galois module isomorphism), then $\alpha$ becomes $\alpha t^2$. 

We thus define the \vocab{symplectic type} of $\phi$ to be the image of $\alpha$ in $\FF_\ell^\times / (\FF_\ell^\times)^2 \simeq \{\pm 1\}$. If the image is $+1$, then $\phi$ has \vocab{positive symplectic type}, and if the image is $-1$, then $\phi$ has \vocab{negative symplectic type}.\footnote{\cite{local_symplectic_type} and \cite{global_symplectic_type} call positive and negative symplectic type symplectic and antisymplectic isomorphism, respectively. We choose a different terminology to avoid confusion.} 
This terminology is not to be confused with symplectic and antisymplectic.
\begin{proposition}
    \label{prop:positive_negative_symplectic}
    Suppose $\ell\equiv 1\pmod 4$ (resp. $\ell\equiv 3\pmod 4$).
    Then there exists an antisymplectic $\GalQ$-module isomorphism $\phi : H^\perp/H\to \Jac(X)[\ell]$ if and only if there exists a $\GalQ$-module isomorphism $\psi : H^\perp/H \to \Jac(X)[\ell]$  of positive (resp. negative) symplectic type.
\end{proposition}
\begin{proof}
    Both facts follow from: $-1\in\mathbb (\FF_\ell^\times)^2$ if and only if $\ell\equiv 1\pmod 4$.
\end{proof}
From \Cref{prop:positive_negative_symplectic}, it suffices to determine whether a $\GalQ$-module isomorphism we have is of positive or negative symplectic type. It is possible that there is a $\GalQ$-module isomorphism of positive symplectic type and a $\GalQ$-module isomorphism of negative symplectic type at the same time. By \cite[Thm.~15]{local_symplectic_type}, this happens if and only if the image of $\bmod\,\ell$-representation $\overline\rho_{X,\ell}(\GalQ)$ is abelian and not the subgroup generated by a conjugate of $\ttwomatrix a10a \in \GL_2(\FF_\ell)$ for some $a\in\FF_\ell^\times$.
\subsection{A Local Test for Symplectic Type}
\label{subsec:symplectic_test_description}
We do not have a general algorithm for determining the symplectic type given $H^\perp/H$ and $X$. However, if there exists $\sigma\in\GalQ$ which acts on $\Jac(X)[\ell]$ by a non-diagonalizable matrix (i.e., conjugate of $\ttwomatrix a10a$ for some $a\in\FF_\ell^\times$), then this element $\sigma$ could be used to determine the symplectic type, as detailed in the following proposition, which is a variant of \cite[Thm.~16]{local_symplectic_type}.
\begin{proposition}
    \label{prop:symplectic_condition}
    Let $\langle\bullet,\bullet\rangle : \FF_\ell^2\times\FF_\ell^2 \to \FF_\ell$ be a non-degenerate alternating bilinear pairing. Let $M\in\GL_2(\FF_\ell)$ be a non-diagonalizable matrix. Then as $v$ varies through $\FF_\ell^2$, then $\langle v, Mv\rangle$ does not depend on $v$ up to multiplication by a square.
\end{proposition}
\begin{proof}
    Without loss of generality, change the basis so that $M = \ttwomatrix a10a$ and, scale the inner product by a constant so that $\left\langle \ttwovector{x}{y}, 
    \left(\begin{smallmatrix} x' \\[-2pt] y' \end{smallmatrix}\right)
    \right\rangle = xy'-yx'$.
    Thus, $\langle v, Mv\rangle = -y^2$, where $v = \ttwovector xy$.
\end{proof}
In the case of the $\ell$-torsion of an elliptic curve $X$, if there exists $\sigma\in\GalQ$ acting on $\Jac(X)[\ell]$ by a non-diagonalizable matrix, then by \Cref{prop:symplectic_condition}, the Weil pairing $e_\ell(P, \sigma(P))$ is either $1$ or does not depend on $P\in\Jac(X)[\ell]$ (up to multiplication by a square in $\FF_\ell^\times$).

Similarly, $\sigma$ acts on $H^\perp/H$ by the same matrix, so by \Cref{prop:symplectic_condition} again, the Weil pairing $e_\ell(Q,\sigma(Q))$ is either $1$ or does not depend on $Q\in H^\perp/H$ (up to multiplication by a square in $\FF_\ell^\times$). If $P = \phi(Q)$, then
\begin{equation}
\label{eq:symplectic_type_relation}
e_\ell(P,\sigma(P)) = e_\ell(\phi(Q), \sigma(\phi(Q)))
= e_\ell(\phi(Q), \phi(\sigma(Q))) = e_\ell(Q,\sigma(Q))^\alpha,
\end{equation}
where $\alpha$ is the symplectic type. Thus, comparing the nontrivial values of $e_\ell(P,\sigma(P))$ and $e_\ell(Q,\sigma(Q))$ for any $P\in\Jac(X)[\ell]$ and $Q\in H^\perp/H$ determines the symplectic type.

To compute this, we note that, by the Chebotarev's density theorem, there exists a prime $p$ for which $\Frob_p$ and $\sigma$ are in the same conjugacy class.
Thus, we may take the reduction of both curves modulo $p$ and consider their equations in $\overline{\FF}_p$.
The following \Cref{lem:small_field} shows that all torsion points are contained in a field extension of degree only $O(\ell^2)$.
\begin{lemma}
\label{lem:small_field}
Suppose that $p$ is a prime such that the action of $\Frob_p$ on $\Jac(X)[\ell]$ (and hence on $H^\perp/H$) is a non-diagonalizable matrix. Then we have
\begin{align}
\Jac(X)_{\overline{\FF}_p}[\ell] &= \Jac(X)_{\FF_{p^{\ell(\ell-1)}}}[\ell] \quad\text{and}\quad  \\
\Jac(Y)_{\overline{\FF}_p}[\ell] &= 
\begin{cases}
\Jac(Y)_{\FF_{p^{\ell(\ell-1)}}}[\ell] & \text{if }\ell\neq 3 \\
\Jac(Y)_{\FF_{p^{18}}}[\ell] & \text{if }\ell = 3 \\
\end{cases}
\end{align}
\end{lemma}
\begin{proof}
    Suppose that the action of $\Frob_p$ on $\Jac(X)_{\overline{\FF}_p}[\ell]$ is conjugate to $\ttwomatrix\gamma 1 0 \gamma$ for some $\gamma\in\FF_\ell$.
    We compute the order of $\Frob_p$ in both torsion fields.
    \begin{itemize}
    \item For $X$, by the condition, $\Frob_p$ is conjugate to $\ttwomatrix \gamma 1 0 \gamma$. The $n$-th power of this is $\ttwomatrix{\gamma^n}{n\gamma^{n-1}}{0}{\gamma^n}$, which is congruent modulo $\ell$ to the identity matrix when $n=\ell(\ell-1)$.
    \item For $Y$, we have that $\Frob_p$ must act on $\Jac(Y)_{\overline{\FF}_p}[\ell]$ by matrix with eigenvalues $\alpha$, $\beta$, $\gamma$, $\gamma$, all in $\FF_\ell$, where $\alpha$ and $\beta$
    are the eigenvalues corresponding to $H$ and $\Jac(Y)_{\overline{\FF}_p}[\ell]/H^\perp$.
    In particular, we deduce that $(\Frob_p)^{\ell-1} - 1$ has all eigenvalues $0$, and hence is a nilpotent matrix. 
    
    In particular, if no Jordan block of $\Frob_p$ has size greater than $\ell$, then
    \begin{equation}
        0 = \big((\Frob_p)^{\ell-1} - 1\big)^\ell = (\Frob_p)^{\ell(\ell-1)} - 1,
    \end{equation}
    where the second equality holds because we are working in modulo $\ell$. 
    The only case that the previous sentence does not cover is when $k=4$ and $\ell=3$ (i.e., $\Frob_p$ is a single Jordan block), in which case the order is $18$.
    \end{itemize}
    Thus, every element in $\Jac(X)_{\overline{\FF}_p}[\ell]$ and $\Jac(Y)_{\overline{\FF}_p}[\ell]$ is fixed by $(\Frob_p)^{\ell(\ell-1)}$ (or $(\Frob_p)^{18}$ for $\Jac(Y)_{\overline{\FF}_p}[\ell]$ and $\ell=3$), completing the proof.
\end{proof}
\begin{remark}
    In the version of \Cref{alg:symplectic} given below, we will consider only the case in which $\alpha$ and $\beta$ are both distinct from $\gamma$, in which case one needs to consider only $\FF_{p^{\ell(\ell-1)}}$ even for $\ell=3$.
\end{remark}
\subsection{Algorithm for Determining Symplectic Type}
With all the tools developed in \Cref{subsec:symplectic_test_description}, we now describe an algorithm to determine the symplectic type on some of the curves.
\begin{algorithm}
\label{alg:symplectic}
\emph{Input.} Two curves $X$ and $Y$ of genus $1$ and genus $2$ for which a one-dimensional Galois-stable $H\subseteq\Jac(Y)[\ell]$ exists and there exists a $\GalQ$-module isomorphism $\phi : H^\perp/H \to \Jac(X)[\ell]$.

\emph{Output.} Either \texttt{positive} or \texttt{negative} symplectic type of $\phi$ or \texttt{fail}, which occurs if the image $\rhoHperpH$ does not contain a non-diagonalizable matrix.

\begin{enumerate}
    \item Find a prime $p$ such that $a_{X,p}^2 \equiv 4p\pmod\ell$. Then check that the matrix $\Frob_p|_{\Jac(X)[\ell]}$ is not diagonalizable by checking that the order of $\Frob_p|_{\Jac(X)[\ell]}$ does not divide $\ell-1$.

    If one cannot find $p$ after a sufficiently many trials, then the image of Galois representation likely does not have a non-diagonalizable element, so return \texttt{fail}.

    Once we find $p$, consider curves $X$ and $Y$ over $\FF_{p^{\ell(\ell-1)}}$.
    \item Pick a random point $P\in\Jac(X)[\ell]$ and then compute $w_1 := e_\ell(P, \Frob_p(P))$. Repeat until this result is not $1$.
    \item Determine the characteristic polynomial of $\Frob_p | _{\Jac(Y)[\ell]}$. Write it in the form $(T-\alpha)(T-\beta)(T-\gamma)^2\in\FF_\ell[T]$ such that $(T-\gamma)^2$ is the characteristic polynomial of $\Frob_p  |_{\Jac(X)[\ell]}$.

    If $\alpha=\beta=\gamma$, repeat (1) again with larger primes.
    \item Pick a random point $R\in\Jac(Y)[\ell]$. Compute $Q = (\Frob_p-\alpha)(\Frob_p-\beta)(R)$. Then compute $w_2 := e_\ell(Q,\Frob_p(Q))$. Repeat until this result is not $1$.
    \item Return \texttt{positive} if $w_1 = w_2^{t^2}$ for some $t\in\FF_\ell^\times$, \texttt{negative} otherwise.
\end{enumerate}
\end{algorithm}
\begin{remark}
    In step (1), by the Chebotaraev density theorem, 
    the density of such $p$ is at least 
    $(\ell-1)/|\GL_2(\FF_\ell)| = \Omega(1/\ell^3)$.
    Since $\ell$ is generally small (practically, $\ell\leq 19$), it is not difficult 
    to obtain $p$ by trying the first few primes.
\end{remark}
\begin{remark}
    One can show that the probability that each attempt of both steps (2) and (4) fails is $\tfrac 1{\ell}$. (For step (4), note that $\alpha\beta = \gamma^2 = p$, so both $\alpha$ and $\beta$ are distinct from $\gamma$.)
\end{remark}
\begin{proposition}
    \Cref{alg:symplectic} correctly determines the symplectic type (if it does not return \texttt{fail}).
\end{proposition}
\begin{proof}
From \labelcref{eq:symplectic_type_relation},
it suffices to show that $Q\in H^\perp$.
Assume without loss of generality that $\Frob_p$ acts on $H$ by multiplication by $\alpha$. Let 
$$ V = \operatorname{Ker}(\Frob_p-\beta)\\(\Frob_p-\gamma)^2.$$
Since $V$ is a span of the three columns corresponding to $\beta$, $\gamma$, $\gamma$ in the Jordan block decomposition, we deduce that $\dim V = 3$.
Moreover, the action of $\Frob_p$ on $H^\perp$ has characteristic polynomial $(T-\beta)(T-\gamma)^2$, so we have $H^\perp\subseteq V$. Comparing dimensions gives $H^\perp = V$. Finally, %we note that
\begin{equation}
    (\Frob_p-\beta)(\Frob_p-\gamma)^2 Q
= (\Frob_p-\alpha)(\Frob_p-\beta)^2(\Frob_p-\gamma)^2 R = 0,
\end{equation}
so $Q\in H^\perp$ as desired.
\end{proof}
This algorithm is generally able to handle $\ell\in\{3,5,7\}$ in a few seconds. 

\section{Gluing Curves}
\label{sec:gluing_algorithm}
Given a genus 2 curve $Y$, we wish to find a prime $\ell$ and a genus 1 curve $X$ to which it can be glued, and then compute the resulting gluing. In this section, we will give a concrete description of our workflow for finding $X$ and computing the gluing. For simplicity, we only look for curves $Y$ such that the Galois representation corresponding to $H^\perp/H$ is irreducible. We then give a demonstration of how the steps pan out on a particular curve.

Our workflow starts with a genus 2 curve $Y$. Because attempting to compute gluings is time-consuming, we first narrow down the possibilities for $\ell$ and $X$ as described in sections \Cref{subsec:determine_ell,subsec:determine_X}. We then attempt to analytically construct gluings for each of the remaining possibilities in \Cref{subsec:computing_gluing}. The implementation of the workflow described in this section is available at \cite{our_code}.

\subsection{Determining \texorpdfstring{$\ell$}{ℓ}}
\label{subsec:determine_ell}
As specified by \Cref{thm-cond:H-stability} we must find $\ell$ for which some $H \subset \Jac(Y)[\ell]$ is Galois-stable. We use \Cref{alg:H_test} to find all possible $\ell$ for which this holds, possibly along with some spurious $\ell$ which do not work.

\subsection{Determining \texorpdfstring{$X$}{X}}
\label{subsec:determine_X}
Once $\ell$ is restricted to some finite set, we fix some $\ell$. We assume that $H$ as in \Cref{thm-cond:H-stability} exists. We wish to find $X$ such that some $\psi$ as in \Cref{thm-cond:phi-equivariance} exists. 

Each of the $\GalQ$-stable subgroups $H$ corresponds to a trace function that takes a prime number $p$ and output $b_p$, the trace of $H^\perp/H$. We use \Cref{alg:frob_trace_recover} to recover all possible trace functions. (If it reports no trace function, this means that no such $H$ exists.) For each such test function $p\mapsto b_p$, we query the LMFDB all elliptic curves $X$ over $\QQ$ that satisfies the conditions of \Cref{prop:frob_traces} for all primes $p\leq 100$. This leaves us with a list of potential elliptic curves. 

Next, given a potential $X$ and $Y$, one can rule out most cases where $\phi$ has the wrong symplectic type by applying \Cref{alg:symplectic}.

Once we have at least one elliptic curve $X$ for which there exists an antisymplectic $\GalQ$-module isomorphism $\phi: H^\perp/H \to \Jac(X)[\ell]$, it is possible to characterize all of them. Indeed, we note that for any elliptic curve $X'$, a symplectic $\GalQ$-module isomorphism $\psi: \Jac(X)[\ell] \to \Jac(X')[\ell]$ gives an antisymplectic $\GalQ$-module isomorphism $\phi' = \psi \circ \phi$. Further, all antisymplectic $\GalQ$-module isomorphisms $\phi': H^\perp/H \to \Jac(X')[\ell]$ arise in this manner, specifically from $\psi = \phi' \circ \phi^{-1}$.

Thus, given such an $X$, the problem of finding all such $X'$ is reduced to finding an elliptic curve $\ell$-congruent to $X$ with positive symplectic type. When $\ell\in\{2,3,5\}$, the moduli space of all such $X'$ is a curve of genus $0$, which has been worked out in \cite{ec_gluable_family_2} and \cite{ec_gluable_family_3_5}. When $\ell\geq 7$, the moduli space of $X'$ is a curve of genus at least $3$. By Falting's theorem, any curve of genus greater than $1$ has finitely many rational points. Thus, for $\ell\geq 7$, there will be only finitely many such $X'$ defined over $\QQ$. Still, the equation for the curve of all possible $X'$ has been worked out for $\ell\in\{7,11\}$ in \cite{ec_gluable_family_7_11}.

\begin{remark}[The reducible case]
    \label{rmk:reducible}
If the Galois representation corresponding to $H^\perp/H$ is reducible, following the above steps above is not enough to filter the list of elliptic curves and reduce it to a list of manageable size. For example, starting with the curve \href{https://www.lmfdb.org/Genus2Curve/Q/961/a/961/1}{\texttt{961.a.961.1}} in the LMFDB and $\ell=5$ leaves $3083$ potential elliptic curves. This is because Frobenius traces can only prove that $\overline\rho_X$ and $\rhoHperpH$ are isomorphic up to semisimplification.

There are two possible strategies that we can use to narrow this list down further. Suppose that $X$ is an elliptic curve such that $\overline\rho_X\simeq\rhoHperpH$.
\begin{enumerate}
\item \textbf{Discriminant.} Suppose that $p$ is a prime not dividing $N_Y$ such that $X$ has a multiplicative reduction modulo $p$. Thus, the Galois representation $\overline\rho_X$ should be unramified at $p$, and so by using the theory of Tate's curve, one can deduce that $\ell$ divides $\nu_p(\Delta_{\min}(X))$ (where $\Delta_{\min}$ denotes the minimal discriminant).
Thus, we can rule out a large number of potential elliptic curves for which this condition does not hold.
\item \textbf{Diagonal Matrix.} Let $p$ be a prime such that $\overline\rho_Y(\Frob_p)$ is a diagonal matrix in $\GSp_4(\FF_\ell)$, then $\overline\rho_X(\Frob_p)$ is also a diagonal matrix in $\GL_2(\FF_\ell)$. 
Thus, one can find such primes $p$ and use them to rule of the elliptic curves further.
\end{enumerate}
Running these two strategies on the curve \href{https://www.lmfdb.org/Genus2Curve/Q/961/a/961/1}{\texttt{961.a.961.1}} (with $12$ primes $1301$, $2351$, $4211$, $5171$, $16001$, $17881$, $24371$, $31181$, $35531$, $36451$, $37361$, and $45751$ in Step (2)) reduces the number of candidates to only 9. However, this does not work very well on some other curves, especially with $\ell=3$, due to the sheer number of candidates and the rarity of primes $p$ in Step (2).
\end{remark}

\subsection{Computing a Gluing}
\label{subsec:computing_gluing}
The computation of gluing is based on the numerical algorithm in \cite[\S 2.1]{1_plus_2_2_torsion_gluing} and the code from \cite{gluing_code}. 
However, we made several optimizations.
First, in the step of computing the curve invariants from lattice (Step (4) in \Cref{alg:fast_gluing} below), we replace the original theta function algorithm with a faster implementation from FLINT \cite{fast_theta}.
Second, the original code repeatedly tests a random maximal isotropic subgroup of $\Jac(X)[\ell]\times\Jac(Y)[\ell]$.
However, this proves to be inefficient since there are $\Theta(\ell^6)$ such subgroups \cite[Cor.~1.20]{1_plus_2_2_torsion_gluing}.

To optimize this, we propose the following two-step approach to search for the desired subgroup.
\begin{enumerate}
    \item Determine which one-dimensional subgroups $H$ are rational.
    \item Determine which antisymplectic isomorphisms $\phi : \Jac(X)[\ell]\to H^\perp/H$ give rise to a gluing.
\end{enumerate}
This gives the following algorithm.
\begin{algorithm}
\label{alg:fast_gluing}
\emph{Input.} An elliptic curve $X$, a genus $2$ curve $Y$, and a prime $\ell$.

\emph{Output.} A (possibly empty) list of all genus $3$ curves $Z$ defined over $\QQ$ such that $\Jac(Z) \sim (\Jac X\times\Jac Y)/G$ for some maximal isotropic subgroup $G\subset \Jac(X)[\ell]\times\Jac(Y)[\ell]$.

\begin{enumerate}
    \item Compute the period matrix of $Y$, giving a basis $\{P_1,P_2,P_3,P_4\}$ of the lattice $\Lambda$ such that $\Jac(Y)\simeq \CC^2/\Lambda$. 
    \item For each ratio $(a_1:a_2:a_3:a_4)\in\mathbb P^3(\FF_\ell)$, test whether the subgroup $H$ generated by the torsion point $\tfrac{a_1}\ell P_1 + \tfrac{a_2}\ell P_2 + \tfrac{a_3}\ell P_3 + \tfrac{a_4}\ell P_4$ is $\GalQ$-stable by using \Cref{alg:calculate_H}.
    
    \item 
    For each $H$ determined to be $\GalQ$-stable in (2), enumerate all antisymplectic isomorphisms $\phi : \Jac(X)[\ell]\to H^\perp/H$.
    
    \item For each such $\phi$, use $(\phi, H)$ to generate the subgroup $G$ according to \cite[Prop.~1.18]{1_plus_2_2_torsion_gluing}.
    Compute a period matrix of the lattice $\Jac(X)\times\Jac(Y)/G$.
    \item Compute its Diximier--Ohno invariants \cite{quartic_reconstruction} or Shioda invariants \cite{hyperelliptic_reconstruction} and test whether they are defined over $\QQ$ or not.
    \item For any set of rational invariants, reconstruct the curve $Z$ using \cite[Alg.~2.21]{1_plus_2_2_torsion_gluing} for plane quartics or \cite{hyperelliptic_reconstruction} for hyperelliptic curves.
\end{enumerate}
\end{algorithm}
In Step (2), there are $\#\mathbb P^3(\FF_\ell) = \ell^3 + \ell^2 + \ell + 1 = \Theta(\ell^3)$ possible $H$'s to check. In Step (3), there are $\#\SL_2(\FF_\ell) = \ell(\ell^2-1) = \Theta(\ell^3)$ isomorphisms to check.
Thus, this algorithm reduces checking $\Theta(\ell^6)$ subgroups to checking at most $\Theta(\ell^3)$ torsion points and isomorphisms, making it much more efficient

We implemented this algorithm in Magma V.2.28-16. The timing of this algorithm against simply enumerating all maximal isotropic subgroups is shown in \Cref{table:gluing_fast_vs_slow}. The timing was taken on CPU 12th Gen Intel i9-12900K (24) @5.100GHz on five different test cases in which one can find a gluing when working with $500$ decimal digits.
Note that the timing only measures time to compute the (Dixmier-Ohno or Shioda) invariants and does not include time to reconstruct the curve.
\begin{table}[ht]
\begin{tabular}{c|ccc}
& $\ell=3$ & $\ell=5$ & $\ell=7$ \\ \hline
Enumerate all subgroups & 144.5 s & 5760 s & {\small (not attempted)} \\
Using \Cref{alg:fast_gluing} & 11.8 s & 71.9 s & 241 s
\end{tabular}
\caption{Time to compute invariants of all possible gluings $500$ digits.}
\label{table:gluing_fast_vs_slow}
\end{table}

\subsection{Example}
We provide a rundown of our algorithms on a particular curve.
\begin{example}
    Let $Y$ be the genus $2$ curve $y^2 + (x^3 + x^2 + x + 1)y = -x^2 - x$ (\href{https://www.lmfdb.org/Genus2Curve/Q/277/a/277/1}{\texttt{277.a.277.1}} in the LMFDB). Running \Cref{alg:H_test} on this curve yields the result $\{3, 5\}$, meaning that the only possible choices of $\ell$ are $3$ and $5$. From now, suppose we are looking for a $(5,5)$-gluing, i.e., $\ell=5$.
\end{example}
    We next run \Cref{alg:frob_trace_recover} on the input $(Y, 5)$. Since the conductor of $Y$ is $277$, which is prime, the algorithm concludes immediately that $\chi$ must be trivial, i.e., $\chi(p)=1$ for all $p$. (This is also reflected by the fact that $Y$ has a rational $5$-torsion subgroup.)
    Thus, we may compute the Frobenius traces $b_p$ of $H^\perp/H$ from the Frobenius traces $a_{p,Y}$ of $Y$. For example,
    \begin{equation}
    \begin{split}
    F_{Y,13}(T) &= T^4 - 3T^3 + 7T^2 - 39 T + 169,\\
    a_{Y,13} &= 3,\quad \text{and}
    \quad b_{13} = 3 - 1 - \tfrac{13}1 = 4\pmod 5.
    \end{split}
    \end{equation}
    
    We can repeat the above process to compute $b_p$ for all primes $p\leq 100$. Then we search for all elliptic curves in the LMFDB satisfying the trace constraints detailed in \Cref{prop:frob_traces}.
    The result is the four curves shown in \Cref{table:gluable_curves}.
    \begin{table}[ht]
    \centering
    \begin{tabular}{cccc}
        & LMFDB & Equation & Symplectic Type \\ \hline
        $X_1$ & 
        \href{https://www.lmfdb.org/EllipticCurve/Q/1939/b/1}{\texttt{1939.b1}}
         & $y^2+y=x^3-1916x-32281$
        & positive\\
        $X_2$
        & \href{https://www.lmfdb.org/EllipticCurve/Q/18559/a/1}{\texttt{18559.a1}}
        & $y^2+y=x^3+11734x-21208$ 
        & negative \\
        $X_3$
        &\href{https://www.lmfdb.org/EllipticCurve/Q/21883/b/1}{\texttt{21883.b1}}
        & $y^2+y=x^3-86x-44420$ 
        & positive \\
        $X_4$
        & \href{https://www.lmfdb.org/EllipticCurve/Q/32963/c/1}{\texttt{32963.c1}}
        & $y^2+y=x^3-77866x+8364065$ 
        & positive \\
    \end{tabular}
    \caption{Potential elliptic curves gluable to curve \href{https://www.lmfdb.org/Genus2Curve/Q/277/a/277/1}{\texttt{277.a.277.1}}}
        \label{table:gluable_curves}
    \end{table}
    
    We run the symplectic test in \Cref{alg:symplectic} to test the antisymplectic condition. Both curves can be tested using $p=19$, and the results are shown in \Cref{table:gluable_curves}. Since $\ell\equiv 1\pmod 4$, by \Cref{prop:positive_negative_symplectic}, the symplectic test rules out $X_2$

    Thus, our algorithm found the curves $X_1$, $X_3$, and $X_4$. This is \emph{not} a proof that these curves are gluable to $Y$. We can either use \Cref{alg:proving_isomorphism} or computing the gluing explicitly to prove that they actually form a Galois stable maximal isotropic subgroup $G$.

    For $i\in\{1,3,4\}$, running the code referenced in \Cref{subsec:computing_gluing} on $X_i$ and $Y$ shows that there are indeed gluings $Z_i$.
    The invariants of $Z_1$ and $Z_3$ was obtained by computing at $500$-digit precision, while the invariants of $Z_4$ was obtained at $1000$-digit precision.
    It took about 30 seconds to compute the coordinates a nonzero point in $H$
    and less than 60 seconds to compute the invariants for each of $Z_1$, $Z_3$, and $Z_4$.
    However, by far dominating every computation we have done is minimizing the equations of $Z_4$, which took just over an hour.
    The minimized equations of $Z_1$, $Z_3$, and $Z_4$ are given below.
    \begin{gather*}
    \medmath{
    \begin{split}
        Z_1: 88189 x^4 &- 398531 x^3 y + 7700 x^3 z - 678120 x^2 y^2 
        + 1444780 x^2 y z + 231034 x^2 z^2  \\
        &+ 238603 x y^3 - 1620885 x y^2 z  
        - 218291 x y z^2 - 420855 x z^3 + 82587 y^4   \\
        &- 2912900 y^3 z + 333537 y^2 z^2 - 959874 y z^3 - 281678 z^4 = 0
    \end{split}
    } \\[4pt]
    \medmath{
    \begin{split}
        Z_3: y^2 = 448x^8 &+ 3584x^7 + 2016x^6 - 476x^5 \\
        &- 13020x^4 - 16408x^3 - 18340x^2 - 8988x - 4025
    \end{split}
    } \\[4pt]
    \medmath{
        \begin{split}
        Z_4:  19351616 x^4 &+ 136748535 x^3 y + 106394158 x^3 z  - 235515177 x^2 y^2 \\
        &- 46043175 x^2 y z + 67674485 x^2 z^2 
        - 549641282 x y^3  + 36999650 x y^2 z \\
        &- 160500711 x y z^2 
        - 36439076 x z^3 + 272167382 y^4 + 488584945 y^3 z \\
        &- 488728851 y^2 z^2 + 152950443 y z^3 - 115190535 z^4 = 0
        \end{split}
    }
    \end{gather*}
\section{Examples}
\label{sec:examples}
We now report on some examples resulting from the search process in \Cref{sec:gluing_algorithm}. As in \Cref{subsec:computing_gluing}, the amount of time to compute a gluing is defined to be the time taken to compute Diximier-Ohno invariants or Shioda invariants. 
It does not include time to reconstruct or minimize the curve.
All timings were done on CPU 12th Gen Intel i9-12900K (24) @5.100GHz.
\subsection{Gluing Along Large Torsion}
Our work allows us to look for gluings along $\ell$-torsion for larger $\ell$.

We first note that for $\ell\in\{7,11,13,17,19,37,43,67,163\}$ 
(i.e., all $\ell\geq 7$ for which there exists an $\ell$-isogeny of elliptic curves, by Mazur's isogeny theorem \cite[Thm.~1]{mazur_isogeny}),
one can construct infinitely many gluable candidates of genus $1$ and genus $2$ curves.
Here, a \vocab{gluable candidate} is a pair of curves $(X,Y)$ such that $X$ has genus $1$, $Y$ has genus $2$, and there exists a Galois stable maximal isotropic subgroup $G\subset \Jac(X)[\ell]\times\Jac(Y)[\ell]$.
As noted in \Cref{rmk:decomposable_jacobian}, a gluable candidate does not necessary produce a gluing.

There is an easy (and uninteresting) way to construct those gluable candidates.
We will construct pairs of elliptic curve $X$ and genus $2$ curve $Y$ such that
\begin{itemize}
    \item $Y$ is isogenous to a product of two elliptic curves $Y_1\times Y_2$;
    \item $Y_1$ has an  $\ell$-isogeny; and 
    \item $\Jac(X)[\ell]$ and $\Jac(Y_2)[\ell]$ are isomorphic as $\GalQ$-modules.
\end{itemize}

Call a gluable candidate $(X,Y)$ \vocab{uninteresting} if it is of the above form and \vocab{interesting} otherwise.

\begin{proposition}
    \label{prop:unintersting_gluing}
    For $\ell\in\{7,11,13,17,19,37,43,67,163\}$, 
    there exists infinitely many uninteresting gluable candidates $(X,Y)$ along $\ell$-torsion.
\end{proposition}
\begin{proof}
    Let $p=2$ if $\ell\equiv 1\text{ or }3\pmod 8$ and $p=3$ if $\ell \in \{7, 13, 37\}$.
    Let $Y_1$ be an elliptic curve with $\ell$-isogeny,
    and let $Y_2$ be an elliptic curve such that $Y_2[p]$ and $Y_1[p]$ are antisymplectically isomorphic as $\GalQ$-modules.
    There are infinitely many such $Y_2$ because
    when $p=2$, one can take elliptic curves
    whose $2$-torsion field is isomorphic to that of $Y_1$,
    and when $p=3$, this follows from \cite[\S 13]{ec_hessian}.
    
    The curves $Y_1$ and $Y_2$ are gluable along $p$-torsion, producing a curve $Y$ 
    for which there exists a $p$-isogeny 
    $\phi : Y_1\times Y_2\to \Jac(Y)$.
    Since $\ell\neq p$, the map $\phi$ induces an isomorphism
    $\phi : \Jac(Y_1)[\ell]\times \Jac(Y_2)[\ell] \stackrel\sim\to \Jac(Y)[\ell]$. This isomorphism is antisymplectic because 
    it has degree $p$ (so by properties of Weil pairing,
    $e_\ell(\phi(P),\phi(Q)) = e_\ell(P,Q)^p$), and by our choice of $\ell$, the number $-p$ is a quadratic residue modulo $\ell$.

Let $X$ be any elliptic curve such that $\Jac(X)[\ell]$ and $\Jac(Y_2)[\ell]$ are isomorphic as $\GalQ$-modules.
In particular, we may take $X=Y_2$.
We now show that $X$ and $Y$ are gluable.
\begin{itemize}
\item Since $Y_1$ has an $\ell$-isogeny, there exists a one-dimensional subgroup $G\subset \Jac(Y_1)[\ell]$.
Then the image $H := \phi(G\times\{0\})$ is a one-dimensional subgroup of $\Jac(Y)$.
\item By Galois equivariance of the Weil pairing, we have $H^\perp/H = \phi(\{0\}\times Y_2)$. Since $\phi$ is antisymplectic, $H^\perp/H$ and $\Jac(Y_2)[\ell]$ are antisymplectic as $\GalQ$-modules. \qedhere
\end{itemize}
\end{proof}
We thus look for interesting gluable candidates. For $\ell\geq 11$, we considered all genus $1$ and $2$ curves in the LMFDB and concluded that there are no interesting gluable candidates in the LMFDB. 
Therefore, we use a larger dataset of genus $2$ curves provided by Sutherland \cite{genus2_extended_db} to obtain the following examples. We check that the gluing is interesting by computing the geometric endomorphism algebra of $Y$ using \cite{endomorphism_ring}.
\begin{example}[Interesting gluing along $13$-torsion]
\label{ex:13_gluing}
Let $X$ be an elliptic curve $y^2+y=x^3+x^2-208x-1256$ (\href{https://www.lmfdb.org/EllipticCurve/Q/75/a/1}{\texttt{75.a1}} in the LMFDB). Let $Y$ be a genus $2$ curve $y^2 + x^3y = -5x^4 + 45x^2 + 9x$, which has conductor $151\,875$ and minimal discriminant $2\,883\,251\,953\,125$.
Our code computes a gluing along $13$-torsion of $X$ and $Y$ in 24 minutes under with $1000$-digit precision, and minimizing the model gives
%$y^2 = 1008x^8 - 4032x^7 + 336x^6 + 8064x^5  + 9660x^4 - 4914x^3 - 7434x^2 - 2478x + 2058$
$42 y^2 = 24 x^8 - 96 x^7 + 8 x^6 + 192 x^5 + 230 x^4 - 117 x^3 - 177 x^2 - 59 x + 49.$
\end{example}
\begin{example}[Interesting gluing along $11$-torsion]
Let $X$ be an elliptic curve $y^2+xy=x^3+9096x+224832$ (\href{https://www.lmfdb.org/EllipticCurve/Q/966/k/1}{\texttt{966.k1}} in the LMFDB). Let $Y$ be a genus $2$ curve $y^2 + (x^2+x)y = x^6 - 3x^5 + 9x^4 - 5x^3 + 12x^2 - 6x$. The resulting gluing when $\ell=11$ is
$$
\medmath{
    \begin{split}
    Z : 39753 x^4 &  + 89236 x^3 y  - 76006 x^3 z  - 3537 x^2 y^2  - 469 x^2 y z   + 46697 x^2 z^2  
    -2200 x y^3   + 42003 x y^2 z \\   
    &+ 29597 x y z^2  - 58478 x z^3 - 4883 y^4  - 12000 y^3 z  - 9287 y^2 z^2  - 398 y z^3   + 6544 z^4 = 0.
    \end{split}
}$$
\end{example}

\subsection{Curves with Interesting Geometric Endomorphism Rings}
\label{subsec:endomorphism}
For any abelian variety $A$, the \vocab{geometric endomorphism ring} $\End(A_{\Qbar})$ is the ring of all endomorphisms $A\to A$ defined over $\Qbar$. For any curve $C$, its geometric endomorphism ring is defined as the geometric endomorphism ring of its Jacobian $\End(C_{\Qbar}) := \End(\Jac C_{\Qbar})$. The \vocab{geometric endomorphism algebra} of a curve $C$ is defined as $\End(C_{\Qbar})\otimes_\ZZ\QQ$.
If $Z$ is a gluing of $X$ and $Y$, then the geometric endomorphism algebra of $Z$ can be easily determined by the geometric endomorphism algebra of $X$ and $Y$.
Furthermore, one can verify the endomorphism algebra by the code in \cite{endomorphism_ring}.

\begin{example}
    Let $X$ be the elliptic curve $y^2=x^3+x^2-3x+1$ (\href{https://www.lmfdb.org/EllipticCurve/Q/256/a/2}{\texttt{256.a2}} in the LMFDB).
    Let $Y$ be the genus $2$ curve $y^2 + y = 6x^5 + 9x^4 - x^3 - 3x^2$ (\href{https://www.lmfdb.org/Genus2Curve/Q/20736/l/373248/1}{\texttt{20736.l.373248.1}} in the LMFDB).

    There are two gluings between $X$ and $Y$ with $\ell=3$, whose minimized equations are given below
    \begin{gather*}
    \medmath{
        Z : y^2 = -210x^7 - 630x^6 + 245x^5 + 1155x^4 - 70x^3 - 700x^2 + 140
    }
    \\
        \medmath{
        \begin{split}
        Z' :
            8 x^4 &- 26 x^3 y + 26 x^2 y^2 + 32 x y^3 - 3 y^4 + 29 x^3 z - 26 x^2 y z + 234 x y^2 z \\
            &- 26 y^3 z + 153 x^2 z^2 + 22 x y z^2 + 17 y^2 z^2 + 65 x z^3 - 250 y z^3 + 25 z^4 = 0
        \end{split}
        }
    \end{gather*}
    One can compute $\End(X_{\Qbar}) \otimes_\ZZ\QQ \simeq \QQ[\sqrt{-2}]$ 
    and $\End(Y_{\Qbar})\otimes_\ZZ\QQ \simeq B_{2,3}$,
    so 
    $$
    \End(Z_{\Qbar})\otimes_\ZZ\QQ \simeq \End(Z'_{\Qbar})\otimes_\ZZ\QQ\simeq \QQ[\sqrt{-2}] \times B_{2,3},$$
    where $B_{2,3}$ is the unique quaternion algebra over $\QQ$ ramified at $2$ and $3$.
\end{example}

% \begin{example}
%     Let $X$ be the elliptic curve $y^2+xy=x^3-x^2-107x+552$ (\href{https://www.lmfdb.org/EllipticCurve/Q/49/a/2}{\texttt{49.a2}} in the LMFDB) and $Y$ be the genus $2$ curve $y^2 + (x^2 + x)y = x^5 + x^4 + 2x^3 + x^2 + x$ (\href{https://www.lmfdb.org/Genus2Curve/Q/686/a/686/1}{\texttt{686.a.686.1}} in the LMFDB).
%     A gluing of $X$ and $Y$ with $\ell=3$ is given by a quartic
%     \begin{equation}
%     Z: 7x^4 + 28x^2z^2 + 24xy^2z + 7y^4 - 4z^4 = 0.\end{equation}
%     The Jacobian $\Jac(Y)$ is isogenous to a product of two elliptic curves, one of which is $X$, and the other is $X'$ in isogeny class \href{https://www.lmfdb.org/EllipticCurve/Q/14/a}{\texttt{14.a}} in the LMFDB. Therefore, $\Jac(Z) \sim X^2 \times X'$.
%     Since are 
%     $\End(X_{\Qbar})\otimes_\ZZ\QQ \simeq \QQ(\sqrt{-7})$ and $\End(X'_{\Qbar})\otimes_\ZZ\QQ \simeq \QQ$, we have
%     \begin{equation}\End(Z_{\Qbar})\otimes_\ZZ\QQ \simeq \operatorname{Mat}_2(\QQ(\sqrt{-7})) \times \QQ.
%     \end{equation}
% \end{example}

\subsection{Gluing Curves in the LMFDB}

We have applied our algorithms to every genus 2 curve $Y$ in the LMFDB to identify triples $(X, Y, \ell)$ such that $X$ and $Y$ are potentially gluable along $\ell$-torsion where one of the following criteria are met:
\begin{itemize}
    \item $\ell = 3$ and $H^\perp/H$ is an irreducible Galois module. (If $H^\perp/H$ is reducible, our filtering methods described in \Cref{rmk:reducible} sult in a large number of false positives due to sheer number of candidate pairs of curves $(X,Y)$.) Our methods identified $3009$ genus 2 curves $Y$ for which at least one candidate $X$ exists.
    \item $\ell \ge 5$. We have identified $704$ pairs of $(Y,\ell)$ for which at least one candidate $X$ exists.
    The distribution by $\ell$ is shown in \Cref{table:gluable_l_stats}.
    All gluings for $\ell\geq 11$ are uninteresting in the sense of \Cref{prop:unintersting_gluing}. Of these pairs, $44$ curve $Y$'s have $H^\perp/H$ reducible, all of which comes from $\ell=5$.
\end{itemize}
\begin{table}[ht]
\begin{tabular}{c|ccccccc}
    $\ell$ & $5$ & $7$ & $11$ & $13$ & $19$ & $37$ & $67$ \\ \hline
    Number of $Y$'s & $649$ & $35$ & $11$ & $2$ & $4$ & $2$ & $1$
\end{tabular}
\caption{Number of $Y$ for which there exists at least one candidate of gluable elliptic curves along $\ell$-torsion.}
\label{table:gluable_l_stats}
\end{table}
For every such $(Y, \ell)$ such that $H^\perp/H$ is irreducible and we can run the symplectic test, we attempt to glue $Y$ to the candidate $X$ which has minimal conductor. For $\ell = 3$ we attempt this even when we cannot run the symplectic test.
The result is shown in \Cref{table:gluing_results}.
In all but two cases that we did not find a gluing, the Jacobian of curve $Y$ splits, in which believe that the cause of failure comes from that the condition (ii) of \Cref{thm:gluability_converse} does not hold (cf. \Cref{rmk:decomposable_jacobian}). One further case comes from a pair with the wrong symplectic type, but the symplectic test did not produce a result.
The only other case that we did not find a gluing is when $\ell=3$, $Y$ is the genus $2$ curve labeled \href{https://www.lmfdb.org/Genus2Curve/Q/471900/a/943800/1}{\texttt{471900.a.943800.1}} in the LMFDB, and $X$ is the elliptic curve \href{https://www.lmfdb.org/EllipticCurve/Q/298800/ff/1}{\texttt{298800.ff.1}} in the LMFDB. 
In this case, the Galois representations $H^\perp/H$ and $X[3]$ are not isomorphic because the former is unramified at $83$ but the latter is not.

The output of our gluings can be found in \cite{our_gluing_results}.
\begin{table}[ht]
\begin{tabular}{c|ccccccc}
    & $\ell=3$ & $\ell=5$ & $\ell=7$ \\ \hline
    \begin{tabular}{c}
        $H^\perp/H$ irreducible and \\
        (for $\ell \geq 5$) passes symplectic test 
    \end{tabular}
    & 3009 & 595  & 32 \\ \hline
    Successful gluing &  2575 & 536 & 19 \\
    Gives an error & 5 & 9 & 1 \\
    Did not find gluing & 429 & 50 & 12
\end{tabular}
\caption{Gluing Results.}
\label{table:gluing_results}
\end{table}
\printbibliography
\end{document}